\documentclass[12pt]{amsart}
\usepackage{amsmath, amsthm, amssymb, verbatim}

\numberwithin{equation}{section} \theoremstyle{plain}
\newtheorem{thm}[equation]{Theorem}    \newtheorem{prop}[equation]{Proposition}
\newtheorem{lem}[equation]{Lemma}      \newtheorem{cor}[equation]{Corollary}
 
\theoremstyle{definition}
\newtheorem{defn}[equation]{Definition}\newtheorem{exa}[equation]{Example}
\theoremstyle{remark}
\newtheorem{rem}[equation]{Remark}
\newtheorem{obs}[equation]{Observation}
\newtheorem*{rem*}{Remark}
\newenvironment{proof*}[1]{\subsubsection*{#1}}{\qed\smallskip}

\def\N{\mathbb{N}}      \def\R{\mathbb{R}}      \def\Z{\mathbb{Z}}
\def\a{\alpha}          \def\b{\beta}           \def\g{\gamma}
\def\d{\delta}          \def\e{\varepsilon}     \def\k{\kappa}
\def\la{\lambda}                \def\t{\theta}

\def\bA{\bar{A}}        \def\bB{\bar{B}}        \def\bC{\bar{C}}
\def\bQ{\bar{Q}}        \def\bga{\bar{\alpha}}  \def\bgb{\bar{\beta}}
\def\bm{\bar{m}}        \def\bs{\bar{s}}        \def\bu{\bar{u}}
\def\bv{\bar{v}}        \def\bw{\bar{w}}
\def\bx{\bar{x}}        \def\by{\bar{y}}        \def\bz{\bar{z}}
\def\bT{\overline{T}\,} \def\BT#1{\overline{T_{#1}}\,}
\def\cF{\mathcal{F}}    \def\tx{\tilde{x}}      
\def\mino{\wedge}       \def\maxo{\vee}

\DeclareMathOperator{\dist}{dist}
\DeclareMathOperator{\len}{len}

\def\({\left(}        \def\){\right)}
\def\hcat(#1){CAT$(#1;*)$}   
\def\hrcat(#1){rCAT$(#1;*)$} 
\def\rcat(#1){rCAT$(#1)$}    
\def\ip#1#2#3{\left<#1,#2;\,#3\right>}
\def\IP<#1,#2;#3>{\ip{#1}{#2}{#3}}

\catcode`\@=11       
\def\rf#1{\@rf{#1}#1:;;}
\def\rfs#1{\@rfs{#1}#1:;;}
\def\rfm#1{\@rfF#1<>;;}

\def\@C{C}\def\@E{E}\def\@F{F}\def\@f{f}\def\@L{L}\def\@O{O}\def\@P{P}
\def\@Q{Q}\def\@R{R}\def\@S{S}\def\@T{T}\def\@X{X}\def\@D{D}\def\@s{s}

\def\@rf#1#2:#3;;{\def\@b{#2}
  \ifx\@b\@C Corollary~\ref{#1}\else%
  \ifx\@b\@E (\ref{#1})\else
  \ifx\@b\@F Fact~\ref{#1}\else%
  \ifx\@b\@f Figure~\ref{#1}\else%
  \ifx\@b\@L Lemma~\ref{#1}\else%
  \ifx\@b\@O Observation~\ref{#1}\else%
  \ifx\@b\@P Proposition~\ref{#1}\else%
  \ifx\@b\@Q Question~\ref{#1}\else%
  \ifx\@b\@R Remark~\ref{#1}\else%
  \ifx\@b\@S Section~\ref{#1}\else%
  \ifx\@b\@T Theorem~\ref{#1}\else%
  \ifx\@b\@X Example~\ref{#1}\else%
  \ifx\@b\@D Definition~\ref{#1}\else%
  \ifx\@b\@s \S\ref{#1}\else
  \ref{#1}\fi\fi\fi\fi\fi\fi\fi\fi\fi\fi\fi\fi\fi\fi}

\def\@rfs#1#2:#3;;{\def\@b{#2}
  \ifx\@b\@C Corollaries~\ref{#1}\else%
  \ifx\@b\@F Facts~\ref{#1}\else%
  \ifx\@b\@f Figures~\ref{#1}\else%
  \ifx\@b\@L Lemmas~\ref{#1}\else%
  \ifx\@b\@O Observations~\ref{#1}\else%
  \ifx\@b\@P Propositions~\ref{#1}\else%
  \ifx\@b\@Q Questions~\ref{#1}\else%
  \ifx\@b\@R Remarks~\ref{#1}\else%
  \ifx\@b\@S Sections~\ref{#1}\else%
  \ifx\@b\@T Theorems~\ref{#1}\else%
  \ifx\@b\@X Examples~\ref{#1}\else%
  \ifx\@b\@D Definitions~\ref{#1}\else
  \ref{#1}\fi\fi\fi\fi\fi\fi\fi\fi\fi\fi\fi\fi}

\def\@rfF<#1>#2;;{\def\@c{#2}
  \@rfs{#1}#1:;;\ifx\@c\empty\else\@rfL:#2;;\fi}

\def\@rfL:#1<#2>#3;;{\def\@b{#2}\def\@c{#3}
  #1\ifx\@b\empty\else\ref{#2}\ifx\@c\empty\else\@rfL:#3;;\fi\fi}

\catcode`\@=12       

\begin{document}
\title[Rough CAT(0) spaces]{Rough CAT(0) spaces}
\author{S.M. Buckley and K. Falk}

\thanks{The authors were supported by Science Foundation Ireland.}

\date{}

\address{Department of Mathematics and Statistics, NUI Maynooth, Maynooth,
Co. Kildare, Ireland}%
\email{stephen.buckley@maths.nuim.ie}

\address{Universit\"at Bremen, FB 3 - Mathematik, Bibliothekstra{\ss}e 1,
28359 Bremen, Germany}%
\email{khf@math.uni-bremen.de}

\begin{abstract}
We investigate various notions of rough CAT(0). These conditions define
classes of spaces that strictly include the union of all Gromov hyperbolic
length spaces and all CAT(0) spaces.

\end{abstract}

\maketitle

\setlength\parskip{\smallskipamount}
\setlength\parindent{0pt}

\section{Introduction}

Gromov hyperbolic spaces and CAT(0) spaces have been intensively studied, in
particular with regard to their boundary theories, which display many common
features as for instance the presence of canonical boundary topologies. It
is thus natural to ask whether there is a `unified theory' including Gromov
hyperbolic spaces, CAT(0) spaces, and more, together with as much common
boundary theory as possible. In this paper we discuss various possible
variants of such a `unified theory' of so-called rough CAT(0) spaces, also
taking into consideration some existing weak notions of nonpositive
curvature. We first investigate properties of the interior of such spaces,
such as the property of having (roughly) unique geodesics, and then produce
non-trivial examples of rough CAT(0) spaces. In a sequel of this paper
\cite{BF}, we investigate the boundary theory within the `unified theory' of
Gromov hyperbolic and CAT(0) spaces introduced here.

There are already a wide variety of conditions related to rough CAT($\k$).
In geometric group theory alone, there are notions such as semi-hyperbolic
groups in the sense of Alonso and Bridson \cite{AB}, and
(bi)automatic~\cite{GS} and (bi)combable groups (for all of which,
see also \cite{Br} and the references therein),
but these notions are much weaker than CAT(0), as are metric geometry
notions such as the Ptolemaic condition (\cite{FLS}, \cite{BFW}), Busemann
convexity \cite{Bu}, $k$-convexity \cite{O}, and $L$-convexity \cite{O}.
Rough CAT(0) is closer to CAT(0) than are any of these conditions, allowing
us to prove for rough CAT(0) spaces (here and in \cite{BF}) analogues of
more parts of the CAT(0) theory than can be proved for any of these other
notions. Another related notion is the CAT(-1,$\e$) notion of Gromov
(\cite{G}, \cite{DG}) which implies Gromov hyperbolicity: this is very
closely related to rough CAT($\k$) for $\k<0$, a notion we briefly consider
and show to be equivalent to Gromov hyperbolicity.

Recall that in the context of geodesic metric spaces, $\delta$-hyperbolic
spaces, $\delta \geq 0$, are spaces with the property that for every
geodesic triangle, each side of the triangle is contained in a
$\delta$-neighborhood of the union of the other two sides.
On the other hand, CAT(0) spaces are geodesic spaces with metric $d$ having
the property that for any two points $u$ and $v$ on a geodesic triangle the
comparison points $\bu$ and $\bv$ in some Euclidean comparison triangle
satisfy $d(u,v) \leq |\bu - \bv|$. It is thus natural to introduce some
amount of `additive fudge' to this comparison property in order to obtain
the notion of a rough CAT(0) space.

We work in length spaces and thus replace geodesic triangles and segments by
$h$-short triangles made of $h$-short segments, which were introduced by
V{\"a}is{\"a}l{\"a} \cite{Va} in the context of Gromov hyperbolicity: a
$h$-short segment, $h\ge0$, is a path whose length is larger by at most $h$
than the distance between its endpoints. Comparison triangles can be defined
using the distances between vertices, and one could then attempt to define a
rough CAT(0) condition by introducing a uniform additive fudge to the CAT(0)
condition. There is, however, a problem with choosing a fixed $h$, since
then even the Euclidean plane would not be rough CAT(0): $h$-short segments
are not forced to remain a uniformly bounded distance apart when the
distance between their common endpoints increases; see Example~\ref{X:big
h}. Thus, $h$ must depend on how far apart are the vertices of a $h$-short
triangle.

Since the definition can be formulated in this generality, we introduce
rough CAT($\k$) spaces with $-\infty \leq \k \leq 0$; the case $\k>0$ is
trivial and we discard it. We write rCAT($\k$) as an abbreviation of ``rough
CAT($\k$)''. We define a notion of rCAT($\k$) spaces with an explicit upper
bound on $h$ which, although useful for many purposes, seems a little
contrived. We therefore also define a variant condition \hrcat(\k), where
the positive upper bound on $h$ is an arbitrary positive function of the
vertices of the triangle. This variant is aesthetically more pleasing, but
turns out to be equivalent to the original notion of rCAT($\k$), a fact that
will prove to be quite useful in \rf{S:examples}.

We also define weak and very weak rCAT($\k$) conditions. The weak rCAT($\k$)
condition, which is equivalent to the full strength CAT($\k$) condition at
least when $k<0$, is equivalent to a certain $4$-point subembedding
condition which makes it clear that it is stable under many limiting
processes. The very weak rCAT(0) condition will be seen to be equivalent to
the bolicity condition of Kasparov and Skandalis \cite{KS1}, \cite{KS2} that
was introduced in the context of their work on the Baum-Connes and Novikov
Conjectures.

Some of the results mentioned above are established in \rf{S:rcatk}, and the
remaining ones are proven in \rf{S:rcat0} where, motivated by the fact that
CAT(0) spaces are uniquely geodesic, we explore a rough unique geodesic
property for (weak) rough CAT(0) spaces; see \rf{T:r-uniq-geo}. We also
prove in \rf{S:rcat0} that every CAT(0) space is $(2+\sqrt{3})$-rCAT(0).

Knowing that the class of rCAT(0) spaces includes both Gromov hyperbolic
length spaces and CAT(0) spaces,
it is natural to ask whether there are rCAT(0) spaces that are neither
CAT(0) nor Gromov hyperbolic. In \rf{S:examples} we give two constructions
(products and gluing) for getting new rCAT(0) spaces from old ones,
which easily produce such examples.

In \rf{T:product-rCAT0} we show that the $l^2$-product of rough CAT(0)
spaces is also rough CAT(0). A rough CAT(0) space that is neither CAT(0) nor
Gromov hyperbolic is thus obtained by taking the $l^2$-product of a Gromov
hyperbolic space that is not CAT(0) and a CAT(0) space that is not Gromov
hyperbolic (e.g.~the $l^2$-product of the unit circle and the Euclidean
plane).

\rf{T:glue1} shows that gluing rough CAT(0) spaces along bounded isometric
subspaces also gives rough CAT(0) spaces, but \rf{X:glue} shows that this
mechanism breaks down as soon as we ask for unbounded gluing sets, even if
they are convex. Finally, \rf{P:VS} shows that normed vector spaces do not
produce interesting examples, since they must be CAT(0) if they are rough
CAT(0).

We wish to thank the referee for carefully reading the paper and for
spotting an error in an earlier version of \rf{X:glue}.

\section{Preliminaries} \label{S:Prelims}

Let $(X,d)$ be a metric space.
We shall not distinguish notationally between paths
$\g:I\to X$, $I\subset\R$, and their images $\g(I)$.
Suppose $(X,d)$ is rectifiably connected. We define the {\em intrinsic metric
associated with $d$} by
$$
l(x,y) :=
  \inf\{ \len(\g)\,:\, \g \text{ is a path in $X$ containing $x,y$}\}\,.
$$
$(X,d)$ is a {\em length space} if $l=d$. A path $\g$ of length
$d(x,y)$ joining $x,y\in X$ is called a {\em geodesic segment}, and is
often denoted $[x,y]$. $(X,d)$ is a {\em geodesic space} if all pairs
of points can be joined by geodesic segments, that is, the above infimum is
always attained.


\begin{defn}\label{D:rough-geodesic}%
A \emph{$h$-short segment}, $h\ge0$, in the length space $(X,d)$ is a path
\\$\g:[0,L]\to X$, $L\ge0$, satisfying
$$
\len(\g) \geq d(\g(0),\g(L)) \geq \len(\g) - h.
$$
We denote $h$-short segments connecting points $x,y \in X$ by $[x,y]_h$. It
is convenient to use $[x,y]_h$ also for the image of this path, so instead
of writing $z=\g(t)$ for some $0\le t\le L$, we often write $z\in[x,y]_h$.
Given such a path $\g$ and point $z=\g(t)$, we denote by $[x,z]_h$ and
$[z,y]_h$ respectively the subpaths $\g|_{[0,t]}$ and $\g|_{[t,L]}$,
respectively; note that both of these are $h$-short segments. We sometimes
write $\g[x,z]$ and $\g[z,y]$ in place of $[x,z]_h$ and $[z,y]_h$ if we need
to specify the short path (or geodesic) of which we are taking a subpath.
\end{defn}

The above notation requires further explanation because of its ambiguity:
given points $x,y$ in a length space $X$, there are always many short
segments $[x,y]_h$ for each $h>0$, so the notation $[x,y]_h$ involves a
choice. When we use this notation in any part of this paper (by a {\it
part}, we mean a definition or a statement or proof of a result), the choice
of such a path does not affect the truth of the underlying statements.
However, all subsequent uses of $[x,y]_h$ in the same part of the paper
refer to the same choice of short segment, and subsequent uses of $[x,z]_h$
and $[z,y]_h$ for $z\in[x,y]_h$ refer to subpaths of this choice of
$[x,y]_h$. Even once we fix $\g = [x,y]_h:[0,L] \to X$, the definitions of
such subpaths $[x,z]_h$ and $[z,y]_h$ may require a choice of $t\in[0,L]$
for which $z=\g(t)$ (since $[x,y]_h$ might not be an arc). The first use of
$[x,z]_h$ or $[z,y]_h$ in any part of the paper involves such a choice, and
all subsequent uses of either $[x,z]_h$ or $[z,y]_h$ in the same part is
consistent with this choice of $t$.

Note that a $0$-short segment is a geodesic segment; in this case, we simply
write $[x,y]$ instead of $[x,y]_0$, and we also write $(x,y)$ for the
subpath of $[x,y]$ with endpoints removed. Geodesic segments are used in
this paper only in the context of the model spaces $M^2_\k$.

\begin{rem}
The fact that $(X,d)$ is assumed to be a length space ensures that for any
$x,y\in X$ and $h>0$, there exists an $h$-short segment $[x,y]_h$.
\end{rem}


Given a number $\k \in \R$, the metric model space $M^2_\k$ is defined as
follows. $M^2_0$ is the Euclidean plane, $M^2_\k$, $\k >0$, is obtained
from the sphere by multiplying the metric with $1/\sqrt{\k}$,
and $M^2_\k$, $\k <0$, is obtained from the hyperbolic plane by multiplying
the metric with $1/\sqrt{-\k}$.
For more details we refer for instance to \cite[Chapter I.2]{BH}.

When $\k=-\infty$, $M^2_\k$ is the union of the real and
imaginary axes of $\R^2$ with the length metric attached.
This is a much smaller space than what $M^2_{-\infty}$ would be if it were
defined as a cone at infinity of the space $M^2_\k$ for $\k\in(-\infty,0)$.
We are, however, only interested in embeddings and subembeddings of three or four
points in our model space, and for these our simple definition of
$M^2_{-\infty}$ suffices.

Since only the case $\k = 0$ will be considered for the bulk of this paper,
the distance between $a,b \in M^2_\k$ is denoted by $|a - b|$, no matter
what value $\k$ has. For $\k > 0$, let $D_\k$ denote the diameter of
$M^2_\k$; for $-\infty\le \k \le 0$, set $D_\k$ to be infinity.

The following result is referred to as Alexandrov's lemma and will be
instrumental for the considerations in Section~\ref{S:rcatk}.

\begin{lem}[Alexandrov's lemma]\label{L:Alex}%
Let $\k\in\R$ and consider distinct points $A$, $B$, $B'$, $C\in M^2_\k$; if
$\k>0$, we assume that
$|B-C|+|C-B'|+|B-A|+|A-B'|<2D_\k$.
Suppose that
$B$ and $B'$ lie on opposite sides of the line $AC$.
(Note that the triangle inequality and the assumption above imply that
$|B-B'| < D_\k$.)

Consider geodesic triangles $T:=T(A,B,C)$ and $T':=T(A,B',C)$. Let $\a$,
$\b$, $\g$ (resp. $\a'$, $\b'$, $\g'$) be the vertex angles of $T$ (resp.
$T'$) at $A,B,C$ (resp. $A,B',C$). Suppose that $\g+\g'\ge\pi$. Then
$$ |B-C| + |C-B'| \le |B-A| + |A-B'|\,.$$
Let $\bT\subset M^2_\k$ be a geodesic triangle with vertices $\bA$, $\bB$,
$\bB'$ such that
$|\bA-\bB|= |A-B|$, $|\bA-\bB'| = |A-B'|$, and
$|\bB-\bB'|=|B-C|+|C-B'|<D_\k$.
Let $\bC$ be the point in $[\bB,\bB']$ with
$|\bB-\bC|=|B-C|$.
Let $\bga,\bgb,\bgb'$ be the vertex angles of $T$
at vertices $\bA, \bB, \bB'$. Then
$$
\bga \ge \a + \a',\quad \bgb \ge \b,\quad \bgb'\ge \b',\quad
|\bA-\bC|\ge |A-C|\,.
$$
Moreover, an equality in any of these implies the equality in the others,
and occurs if and only if $\g+\g'=\pi$.
\end{lem}

A {\em geodesic triangle} $T(x,y,z)$ in a geodesic space $X$ is a collection
of three points $x,y,z\in X$ together with a choice of geodesic segments
$[x,y]$, $[x,z]$ and $[y,z]$.
Given such a geodesic triangle $T(x,y,z)$, a {\em
comparison triangle} is a geodesic triangle in $M^2_\k$,
$T(\bx,\by,\bz)$, such that corresponding distances
coincide: $d(x,y)=|\bx-\by|$, $d(y,z)=|\by-\bz|$, $d(z,x)=|\bz-\bx|$.
A point $\bu \in [\bx,\by]$ is a \emph{comparison point} for $u \in [x,y]$
if $d(x,u)=|\bx-\bu|$.

For details on the definition and characterizations of CAT($\k$) we refer
the reader for instance to \cite[Chapter II.1]{BH}. Let $X$ be geodesic and
$\k \in \R$. Let $T(x,y,z)$ be a geodesic triangle in $X$ with perimeter
less than $2 D_\k$, and consider a comparison triangle $T(\bx,\by,\bz)$ for
$T(x,y,z)$ in $M^2_\k$. We say that $T(x,y,z)$ satisfies the \emph{CAT($\k$)
condition} if for any $u,v \in T(x,y,z)$,
$$
d(u,v) \leq |u-v|.
$$
In the case that $\k \leq 0$ we call $X$ a \emph{CAT($\k$) space} if
all geodesic triangles in $X$ satisfy the CAT($\k$) condition. For
$\k>0$ we say $X$ is a CAT($\k$) space if all geodesic triangles
of perimeter less than $2D_\k$ satisfy the CAT($\k$) condition.
Equivalently, $u$ can be assumed to be one of the vertices of the
triangle $T(x,y,z)$ and $v$ can be assumed to be on the opposite side.
Even more, $v$ can be assumed to be a midpoint of the opposing side.

Another way of characterizing geodesic CAT($\k$) spaces, $\k \in \R$, is by
using the so-called $4$-point condition. Suppose $x_i\in X$ and $\bx_i\in
M^2_\k$ for $0\le i\le 4$, with $x_0=x_4$ and $\bx_0=\bx_4$. We say that
$(\bx_1,\bx_2,\bx_3,\bx_4)$ is a \emph{subembedding} of $(x_1,x_2,x_3,x_4)$
in $M^2_\k$ if
\begin{align*}
&d(x_i,x_{i-1}) =   |\bx_i-\bx_{i-1}|\,, \quad 1\le i\le 4\,, \\
&d(x_1,x_3)     \le |\bx_1-\bx_3| \quad\mathrm{and}\quad
d(x_2,x_4)     \le |\bx_2-\bx_4|\,.
\end{align*}
The metric space $(X,d)$ satisfies the \emph{$4$-point condition},
if every 4-tuple in $X$ has a subembedding in $M^2_\k$.
When $X$ is geodesic, this turns out to be equivalent to $X$ being CAT($\k$).

Also, $X$ is CAT(0) if and only if the \emph{CN inequality} of Bruhat and Tits is
satisfied, that is, for all $x,y,z \in X$ and all $m \in X$ with
$d(y,m)=d(m,z)=d(y,z)/2$,
$$
d(x,y)^2 + d(x,z)^2 \geq 2d(x,m)^2 + \frac{1}{2}d(y,z)^2.
$$

We refer the reader to \cite{GH}, \cite{CDP}, \cite{Va}, or
\cite[Part III.H]{BH} for the theory of Gromov hyperbolic spaces.
A metric space $(X,d)$ is {\it $\d$-hyperbolic}, $\d\ge 0$, if
$$ \ip xzw\ge \ip xyw\mino \ip yzw - \d\,, \qquad x,y,z,w\in X\,, $$
where $\ip xzw$ is the Gromov product defined by
$$ 2\ip xzw = d(x,w) + d(y,w) - d(x,y)\,. $$

The following is a version of the well-known {\it Tripod Lemma}, almost as
stated in \cite[2.15]{Va}, the only minor difference being that it is stated
for short arcs rather than short paths.

\begin{lem}\label{L:Tripod}
Suppose that $\g_1$ and $\g_2$ are unit speed $h$-short paths from $o$ to
$x_1$ and $x_2$, respectively, in a $\d$-hyperbolic space. Let $y_1=\g_1(t)$
and $y_2=\g_2(t)$ for some $t\ge 0$, where $d(o,y_1)\le\IP<x_1,x_2;o>$. Then
$d(y_1,y_2)\le 4\d+2h$.
\end{lem}

A map $f:(X,d_X)\to(Y,d_Y)$ is an
\emph{$(A,B)$-quasi-isometry} if there are constants $A>0$, $B\ge 0$ such
that
$$
\frac{1}{A}d_X(x_1,x_2) - B \leq d_Y(f(x_1),f(x_2))
\leq A \, d_X(x_1,x_2) +B
$$
for any $x_1,x_2 \in X$, and such that $\dist(y,f(X))\le B$, $y\in Y$.
Here, $\dist(x, A) := \inf \{ d(x,y): y \in A\}$ is the distance of
a point $x$ from a set $A$.
A \emph{$B$-rough isometry} is a $(1,B)$-quasi-isometry; $B$ is called the
\emph{roughness constant} of $f$.

We write $A\mino B$ and $A\maxo B$ for the minimum and maximum,
respectively, of two numbers $A,B$.


\section{Rough CAT($\k$) spaces: basic results} \label{S:rcatk}

In this section, we define rough CAT(0) spaces and some weaker variants of
them, and prove some some basic results involving these conditions.

\begin{defn}\label{D:comp-tri}%
A \emph{$h$-short triangle} $T:=T_h(x_1,x_2,x_3)$ with vertices
$x_1,x_2,x_3\in X$ is a collection of $h$-short segments $[x_1,x_2]_h$,
$[x_2,x_3]_h$ and $[x_3,x_1]_h$. Given such a $h$-short triangle $T$, a
\emph{comparison triangle} will mean a geodesic triangle
$\bT:=T(\bx_1,\bx_2,\bx_3)$ in the comparison space $M^2_\k$,
$-\infty\le\k<\infty$ such that $|\bx_i-\bx_j|=d(x_i,x_j)$,
$i,j\in\{1,2,3\}$. Furthermore, we say that $\bu\in \bT$ is a
\emph{comparison point} for $u\in T$, say $u\in [x_1,x_2]_h$, if
$$
|\bx-\bu| \leq \len([x,u]_h)
  \quad\text{and}\quad
|\bu-\by| \leq \len([u,y]_h)\,.
$$
If $\k=-\infty$, the comparison triangle is called a {\it comparison
tripod}.
\end{defn}

Note that $\bu$ is not uniquely determined by $u$ as in the case of
comparison points for triangles in CAT($\k$) spaces. Also, it immediately
follows from the definition that
$$
|\bx-\bu| \geq \len([x,u]_h) - h
  \quad\text{and}\quad
|\bu-\by| \geq \len([u,y]_h) - h.
$$
In order to avoid cluttered notation, we do not specify the comparison space
in the notation $T(\cdot,\cdot,\cdot)$; the space will always be clear from
the context.

\begin{rem}
Clearly, we can always find comparison triangles in $M^2_\k$ for any
$h$-short triangle $T_h(x,y,z)$ in any length space $X$, as long as
$d(x,y)+d(y,z)+d(z,x)\le 2D_\k$. In fact this amounts to the well-known fact
that triangles in $M^2_\k$ can be constructed with arbitrary sidelengths
$a\le b\le c$, as long as the perimeter $a+b+c$ is at most twice the
diameter of $M^2_\k$ and the triangle inequality $c\le a+b$ holds.
\end{rem}

Recall that a CAT($\k$) space is a geodesic space in which the distance
between any pair of points in a geodesic triangle is at most as large as the
distance between comparison points in a comparison triangle in $M^2_\k$. The
natural definition of rough CAT($\k$) should therefore involve a similar
distance inequality between an arbitrary pair of points in an $h$-short
triangle, and a pair of comparison points in an comparison triangle, for
some $h>0$. Our definition will indeed have this form (and we can work with
length spaces rather than geodesic spaces), but for $\k=0$ (the main case
that interests us!), the value of $h$ must depend on how far apart are the
vertices of the $h$-short triangle. The following example shows that a fixed
$h>0$ ``would not work'' when $\k=0$ in the sense that even the Euclidean
plane would fail to satisfy such a condition.

\begin{exa}\label{X:big h}
Let $h>0$ be fixed, and take $x$, $y=z$ to be the points given in coordinate
form as $(-R,0)$ and $(R,0)$, respectively, for some $R>0$. Let
$T:=T_h(x,y,z)$ be the short triangle consisting of the pair of line
segments from $x$ to $y$ and $y$ to $z$ (the latter being degenerate), plus
a path from $z$ to $x$ consisting of the two line segments from $z$ to
$u:=(0,t)$ and $u$ to $x$, where $t=\sqrt{h R+h^2/4}$; it is clear that $T$
is an $h$-short triangle. The comparison triangle $\bT$ is the (geodesic)
planar triangle with the same vertices. If we take $v$ to be the origin,
then $t$, and so $d(u,v)$, tends to infinity as $R$ tends to infinity, while
the distance between any comparison points in $\bT$ remains bounded.
\end{exa}

\begin{defn}\label{D:rCATk cond}%
Let $-\infty\le\k\le 0$, $C>0$, and $h\ge0$. Suppose $(X,d)$ is a length
space and that $T_h(x,y,z)$ is a $h$-short triangle in $X$. We say that
$T_h(x,y,z)$ satisfies the \emph{$C$-rough CAT($\k$) condition} if given a
comparison triangle $T(\bx,\by,\bz)$ in $M^2_\k$ associated with
$T_h(x,y,z)$, we have
$$ d(u,v) \leq |\bu - \bv| + C\,, $$
whenever $u,v$ lie on different sides of $T_h(x,y,z)$ and $\bu,\bv \in
T(\bx,\by,\bz)$ are corresponding comparison points.
\end{defn}

We define a {\it short function (for a metric space $X$)} to be any function
$H:X\times X\times X\to (0,\infty)$.

\begin{defn}\label{D:rCATk}%
Let $-\infty\le\k\le 0$ and $C>0$. We say that a length space $(X,d)$ is
\emph{$C$-rough \rcat(\k,*)}, or simply \emph{$C$-\hrcat(\k)} if there
exists a short function $H$ such that the following condition holds: if
$T_h(x,y,z)$ is a $h$-short triangle in $X$ for $h:=H(x,y,z)$, then
$T_h(x,y,z)$ satisfies the \emph{$C$-rough CAT($\k$) condition}.
\end{defn}

It is often useful to use a specific short function $H$, so we say that $X$
is $C$-\rcat(\k) if it is $C$-\hrcat(\k) with {\it standard short function}
$H$ defined by
$$ H(x,y,z) = \frac{1}{1\maxo d(x,y)\maxo d(x,z)\maxo d(y,z)} \,. $$
For both the \hrcat(\k) or \rcat(\k) conditions, we call the associated
parameter $C$ the \emph{roughness constant}; we omit this parameter if its
value is unimportant.

In the Euclidean plane, it follows from \rf{X:big h} that $h$ can be no
larger than some multiple of $1/(d(x,y)\maxo d(x,z)\maxo d(y,z))$ in order
for a $T_h(x,y,z)$ to satisfy a given rough CAT(0) condition. It is also
easy to see that any given rough CAT(0) condition requires that $h$ to be
bounded, regardless of how close together $x,y,z$ are. These considerations
show that our definition of a standard short function gives a short function
that is in general as large as it can possibly be, modulo multiplication by
a fixed constant, if we want all CAT(0) spaces to be rCAT(0).

In spite of this justification, our choice of standard short function still
seems a little contrived, and the definition of \hrcat(\k) spaces seems more
natural than that of \rcat(\k) spaces. However, we will show in
\rf{C:infinitesimal} that these two classes are equivalent, with
quantitative dependence of roughness constants.

Trivially if $X$ is $C$-\hrcat(\k) with a given short function $H$, it
is $C$-\hrcat(\k) with any other pointwise smaller short function $G$, so we
may always assume that the short function $H$ is pointwise no larger than
the standard short function.

In the above definitions, we could allow $\k$ to be positive, as long as we
restrict $x,y,z$ so that $d(x,y)+d(y,z)+d(z,x)<2D_\k$ (as in the definition
of CAT($\k$) for $\k>0$). However, it is trivial that every length space is
$C$-rough CAT($\k$) for $C>D_\k$, so the class of all rough CAT($\k$) spaces
is of no interest. For this reason, we insist that $-\infty\le\k\le 0$ from
now on.

It is well known (and easily shown) that the CAT($\k$) condition is
equivalent to a weaker version of the same definition where the comparison
inequality is assumed only when one point is a vertex, and one can even
restrict the other point to being a midpoint of a side. This leads us to the
following definitions.

\begin{defn}\label{D:wrCATk}%
Let $-\infty\le\k\le 0$ and $C>0$. A \emph{weak $C$-rough CAT($\k$)
condition} is similar to the $C$-rough CAT($\k$) condition defined in
\rf{D:rCATk}, except that it is required to hold only when $v=x$ and
$u\in[y,z]_h$. A \emph{very weak $C$-rough CAT($\k$) condition} is also
similar to the $C$-rough CAT($\k$) condition, except that it is required to
hold only when $v=x$ and $u \in [y,z]_h$ is a \emph{$h$-midpoint} of
$[y,z]_h$, that is, if it has the property that the Euclidean midpoint $\bu$
of $[\by,\bz]$ is a comparison point for $u$. \emph{Weak} and \emph{very
weak $C$-\hrcat(\k) spaces}, are then defined by making the associated
changes to the above definitions of $C$-\hrcat(\k) spaces, and we can
analogously define weak and very weak \rcat(\k) spaces.
\end{defn}

By elementary geometry, we see that if $x,y,z$ are points in the Euclidean
plane and $u$ lies on the line segment from $y$ to $z$ with $|u-y|=t|z-y|$,
then
\begin{equation}\label{E:plane CAT0}
(d(x,u))^2 \le (1-t)(d(x,y))^2+t(d(x,z))^2-t(1-t)(d(y,z))^2\,.
\end{equation}
It follows that the weak $C$-rCAT(0) condition can be written in the
following more explicit form: if $u=\la(s)$, where $\la:[0,L]\to X$ is a
$h$-short path from $y$ to $z$ parametrized by arclength, $h$ satisfies the
usual bound, and we have both $td(y,z)\le s$ and $(1-t)d(y,z)\le L-s$ for
some $0\le t\le 1$, then
\begin{equation}\label{E:wrCAT0}
(d(x,u)-C)^2 \le (1-t)(d(x,y))^2+t(d(x,z))^2-t(1-t)(d(y,z))^2\,.
\end{equation}
The very weak $C$-rCAT(0) condition can be written in a similar form, but
with the restriction $t=1/2$.

The following result summarizes what we can say about the relationships
between all these variants of rCAT($\k$) spaces.

\begin{thm}\label{T:various rcat0}
\mbox{}
\begin{enumerate}
\item For $-\infty \le\k<0$, the classes of rCAT($\k$), \hrcat(\k), weak
    rCAT($\k$), and weak \hrcat(\k) spaces all coincide with the class
    of Gromov hyperbolic spaces, and all containment implications hold
    with quantitative dependence of parameters.
\item The classes of rCAT(0) spaces and \hrcat(0) spaces coincide, again
    with quantitative control of parameters, and the same is true of
    weak rCAT(0) and weak \hrcat(0) spaces.
\item The class of rCAT(0) spaces is strictly larger than the union of
    the classes of Gromov hyperbolic and CAT(0) spaces.
\end{enumerate}
\end{thm}

Part (a) of this theorem follows from \rf{T:rough 4-pt} below, while part
(b) follows from \rf{C:infinitesimal}, and an example to prove part (c) was
given in the Introduction (see also \rf{S:examples}).

There are a few other possible relationships between these variant
rCAT($\k$) spaces whose truth we cannot determine. Specifically we do not
know if very weak rCAT($\k$) spaces are necessarily weak rCAT($\k$) (either
for $k<0$ or $k=0$), and we do not know if weak rCAT(0) spaces are
necessarily rCAT(0). While the class of rCAT(0) spaces is the main focus of
our interest in this paper, the (weak) rCAT($\k$) characterization of Gromov
hyperbolicity in the above theorem may also be of some interest.

We now wish to discuss another connection to existing notions of
non-positive curvature. In their work on the Baum-Connes and Novikov
Conjectures, Kasparov and Skandalis \cite{KS1}, \cite{KS2} introduced the
class of bolic spaces which, as our class of rCAT(0) spaces, includes both
Gromov hyperbolic spaces and CAT(0) spaces. It turns out that in the case of
length spaces, bolicity is equivalent to very weak rCAT(0). To see this, we
first note that by work of Bucher and Karlsson \cite{BK}, bolicity is
reduced to a condition reminiscent of the CN inequality of Bruhat and Tits
(\cite[p. 163]{BH} and \cite{BT}).

\begin{defn}\label{D:bolic}%
A metric space $X$ is called \emph{$\delta$-bolic}, for some $\delta > 0$, if
there is a map $m: X \times X \to X$ with the property that for all
$x,y,z \in X$
$$
2 d(m(x,y),z) \leq \sqrt{2d(x,z)^2 + 2d(y,z)^2 - d(x,y)^2} + 4\delta.
$$
\end{defn}

\begin{prop}
\label{P:bolic-vwrCAT(0)}
Let $X$ be a length space. If $X$ is very weak $C$-rCAT(0), $C>0$, then it
is $\delta$-bolic with $\delta=C/2$. If $X$ is $\delta$-bolic, $\delta>0$, then
it is very weak $C$-rCAT(0) with $C=4\delta + \sqrt{2}$.
\end{prop}

\begin{proof}
Let $X$ be a very weak $C$-rCAT(0) space.
Let $x,y,z \in X$ and let $T_h(x,y,z)$
be some $h$-short triangle with comparison triangle $T(\bx,\by,\bz)$.
Let $m(y,z)$ be some $h$-midpoint of $[y,z]_h$. This defines a map
$m: X \times X \to X$. By definition, the Euclidean midpoint $\bm$ of $[\by,\bz]$
is a comparison point for $m(y,z)$. Using the comparison triangle property,
the Euclidean parallelogram law and the very weak $C$-rCAT(0) condition, we obtain
\begin{eqnarray*}
d(x,y)^2 + d(x,z)^2
& = &
|\bx-\by|^2 + |\bx-\bz|^2 \\
& = &
2\,|\bx-\bm|^2 + \frac{1}{2}|\by-\bz|^2 \\
& \geq &
2(d(x,m)-C)^2 + \frac{1}{2} d(y,z)^2.
\end{eqnarray*}
Thus $X$ is $C/2$-bolic.

Let now $X$ be a $\delta$-bolic length space with some $\delta>0$.
Let $T_h(x,y,z)$ be some $h$-short triangle and $T(\bx,\by,\bz)$ a
corresponding comparison triangle in the Euclidean plane. Furthermore,
let $m$ be some $h$-midpoint for $[y,z]_h$, that is, $m$ admits the
Euclidean midpoint $\bm$ of $[\by,\bz]$ as a comparison point.
By definition we thus obtain that $d(y,m)\leq d(y,z)/2 + h$ and
$d(m,z)\leq d(y,z)/2 + h$. By applying the bolic inequality for
$y,z,m \in X$ and $m(y,z) \in X$, and the fact that
$h=1/(1\maxo d(x,y)\maxo d(x,z)\maxo d(y,z))$, it follows that
\begin{eqnarray*}
2 d(m(y,z),m)
& \leq &
\sqrt{2 d(y,m)^2 + 2 d(m,z)^2 - d(y,z)^2} + 4\delta \\
& \leq &
\sqrt{4(d(y,z)/2 +h)^2 - d(y,z)^2} + 4\delta \\
& \leq &
2\sqrt{2} + 4\delta.
\end{eqnarray*}
Applying bolicity for $x,y,z \in X$ and $m(y,z) \in X$ now yields
\begin{eqnarray*}
2 d(x,m)
& \leq &
2 d(x,m(y,z)) + 2 d(m(y,z),m) \\
& \leq &
\sqrt{2 d(x,y)^2 + 2 d(x,z)^2 - d(y,z)^2} + 8\delta + 2\sqrt{2}.
\end{eqnarray*}
By using the comparison triangle property and the Euclidean parallelogram
equality we finally deduce
\begin{eqnarray*}
2(d(x,m)-4\delta - \sqrt{2})^2
& \leq &
d(x,y)^2 + d(x,z)^2 - \frac{1}{2} d(y,z)^2 \\
& = &
|\bx-\by|^2 + |\bx-\bz|^2 - \frac{1}{2}|\by-\bz|^2 \\
& = &
2\,|\bx-\bm|^2,
\end{eqnarray*}
which implies the very weak $C$-rCAT(0) inequality with $C=4\delta + \sqrt{2}$.
\end{proof}

The CAT($\k$) condition (for geodesic spaces $X$) is normally stated as the
$C=h=0$ variant of our rCAT($\k$) definition, but it can also be written as
a so-called $4$-point condition. We prove an rCAT($\k$) analogue of
this, but first we need a simple lemma.

\begin{lem}\label{L:r-uniq-geo} Let $x,y$ be a pair of points in the
Euclidean plane $\R^2$, with $l:=|x-y|>0$. Fixing $h>0$, and writing
$L:=l+h$, let $\g:[0,L]\to \R^2$ be a $h$-short segment from $x$ to $y$,
parametrized by arclength. Then there exists a map $\la:[0,L]\to[x,y]$ such
that $\la(0)=x$, $\la(L)=y$, and
\begin{align}
&&|\la(t)-x|&\le |\g(t)-x|\,, \qquad && 0\le t\le L\,, \label{E:laga1}\\
&&|\la(t)-y|&\le |\g(t)-y|\,, \qquad && 0\le t\le L\,, \label{E:laga2}\\
&&\d(t):=\dist(\g(t),\la(t)) &\le M := \frac{1}{2}\sqrt{2lh+h^2}\,,
  && 0\le t\le L\,. \label{E:r-uniq-geo}
\end{align}
In particular if $h\le 1/(1\maxo l)$, then $\d(t)\le \sqrt 3/2$ for all
$0\le t\le L$.
\end{lem}

\begin{proof}
The desired result is invariant under isometries of the plane, so we choose
the points $x=(x_1,0)$ and $y=(y_1,0)$ to be located on the first coordinate
axis with $x_1<y_1$. We also write $\g=(\g_1,\g_2)$ in Euclidean
coordinates. Now define $\la(t)=(\la_1(t),0)$, where $\la_1(t)=(\g_1(t)\maxo
x_1)\mino y_1$. It is clear that $\la(0)=x$, $\la(L)=y$, and that $\la$
satisfies \rf{E:laga1} and \rf{E:laga2}.

It is clear that to maximize $\d(t):=\dist(\g(t),\la(t))$, we should pick
$\g$ to be the concatenation of two straight line paths, one from $x$ to
$\g(t)$ of length $t$ and one from $\g(t)$ to $y$ of length $L-t$. But then
$\g(t)$ traces out an ellipse and it is routine to verify that $\d(t)\le
\d(L/2)=M$.
\end{proof}

\begin{defn}\label{D:rough subembed}%
Let $(X,d)$ be a metric space, $-\infty\le\k\le 0$, and $C\ge 0$. Suppose
$x_i\in X$ and $\bx_i\in M^2_\k$ for $0\le i\le 4$, with $x_0=x_4$ and
$\bx_0=\bx_4$. We say that $(\bx_1,\bx_2,\bx_3,\bx_4)$ is a {\it $C$-rough
subembedding} of $(x_1,x_2,x_3,x_4)$ in $M^2_\k$ if
\begin{align*}
d(x_i,x_{i-1}) &=   |\bx_i-\bx_{i-1}|\,, \quad 1\le i\le 4\,, \\
d(x_1,x_3)     &\le |\bx_1-\bx_3|\,,\qquad\mathrm{and}\\
d(x_2,x_4)     &\le |\bx_2-\bx_4|+C\,.
\end{align*}
\end{defn}

\begin{defn}\label{D:rough 4-pt}%
A metric space $(X,d)$ satisfies the {\it $(C,\k)$-rough 4-point condition},
where $C\ge 0$ and $-\infty\le\k\le 0$, if every 4-tuple in $X$ has a
$C$-rough subembedding in $M^2_\k$. When $\k=0$, we omit it from the
notation.
\end{defn}


\begin{thm}\label{T:rough 4-pt}
For a length space $(X,d)$ and $-\infty\le\k\le 0$, the following conditions
are equivalent, with quantitative dependence of parameters:
\begin{enumerate}
\item $X$ is weak $C$-\rcat(\k) for some $C>0$;
\item $X$ is weak $C$-\hrcat(\k) for some $C>0$;
\item $X$ satisfies the $(C',\k)$-rough 4-point condition for some
    $C'>0$.
\end{enumerate}
Moreover if $-\infty\le\k<0$, then these conditions are quantitatively
equivalent to $\d$-hyperbolicity and to \rcat(\k).
\end{thm}

\begin{proof}
Trivially (a) implies (b), with the same $C$. We next prove that (b) implies
(c) with $C'=2C$. Suppose $X$ is a weak $C$-\hrcat(\k) space with short
function $H$. Let $(x_1,x_2,x_3,x_4)$ be a 4-tuple in $X$, and let
$h:=H(x_1,x_3,x_2)\mino H(x_1,x_3,x_4)$. Choose $h$-short triangles
$T_2:=T_h(x_1,x_3,x_2)$ and $T_4:=T_h(x_1,x_3,x_4)$, and comparison
triangles $\bT_1:=T(\bx_1,\bx_3,\bx_2)$ and $\bT_2:=T(\bx_1,\bx_3,\bx_4)$,
such that $\bT_2$ and $\bT_4$ have a common side $[\bx_1,\bx_3]$, and that
$\bx_2$ and $\bx_4$ lie on opposite sides of the line through $\bx_1$ and
$\bx_3$. Let $\bz$ be the point of intersection of $[\bx_2,\bx_4]$ and the
line through $\bx_1$ and $\bx_3$.

Suppose first that $\bz\in[\bx_1,\bx_3]$; this is always the case if
$\k=-\infty$ but it may fail for finite $\k$. Picking a $h$-short segment
$[x_1,x_3]_h$, let $z\in[x_1,x_3]_h$ be such that $d(x_1,z)=|\bx_1-\bz|$;
note that for $i=2,4$, the points $\bz,\by_i$ are comparison points in the
triangle $\bT_i$ for $z,y_i$, respectively. By the triangle and weak rough
CAT($\k$) inequalities,
$$
d(x_2,x_4) \le d(x_2,z)+d(z,x_4) \le |\bx_2-\bz|+|\bz-\bx_4| + 2C =
  |\bx_2-\bx_4| + 2C\,.
$$
Note that $d(x_1,x_3)=|\bx_1-\bx_3|$. Thus $(\bx_1,\bx_2,\bx_3,\bx_4)$ is a
$C'$-rough subembedding in $M^2_\k$ of $(x_1,x_2,x_3,x_4)$, with $C'=2C$.

Alternatively suppose that the segments $[\bx_1,\bx_3]$ and $[\bx_2,\bx_4]$
do not intersect (and so $\k\in\R$). Let $Q$ be the quadrilateral consisting
of the union of the four geodesic segments $[\bx_1,\bx_2]$, $[\bx_2,\bx_3]$,
$[\bx_3,\bx_4]$, and $[\bx_4,\bx_1]$. Then $M^2_\k\setminus Q$ has two
components: we call the one containing $(\bx_1,\bx_3)$ the {\it inner
component}, and we define the {\it inner and outer angles} at the vertices
in $Q$ in the natural way; note that the inner angle at $\bx_i$, $i=1,3$, is
the sum of the angles at the same point in the triangles $\bT_2$ and
$\bT_4$. If both inner angles were less than $\pi$, then it would follow by
continuity that there exists a point $\bu\in[\bx_1,\bx_3]$ such that the
inner angle at $u$ for the quadrilateral with vertices
$\bx_1,\bx_2,\bu,\bx_4$ is $\pi$. But then it follows from Alexandrov's
Lemma (\rf{L:Alex}) that $\bz=\bu$, contradicting the fact that
$\bz\notin[\bx_1,\bx_3]$.

Thus we may assume without loss of generality that the inner angle at
$\bx_1$ in $Q$ is at least $\pi$. It follows that there exists a geodesic
triangle in $M^2_\k$ with vertices $\tx_2=\bx_2$, $\tx_3=\bx_3$, and
$\tx_0\equiv\tx_4$, together with a point $\tx_1\in [\tx_2,\tx_4]$, such
that $|\tx_i-\tx_{i-1}|=d(x_i,x_{i-1})$, $1\le i\le 4$, and
$$
|\tx_2-\tx_4| = |\tx_2-\tx_1| + |\tx_1-\tx_4| = d(x_2,x_1) + d(x_1,x_4) \ge
  d(x_2,x_4)\,.
$$
By Alexandrov's lemma, $|\tx_1-\tx_3|\ge|\bx_1-\bx_3|=d(x_1,x_3)$.  Thus
$(\tx_1,\tx_2,\tx_3,\tx_4)$ is a $0$-rough subembedding in $M^2_\k$ of
$(x_1,x_2,x_3,x_4)$. Putting this case together with the first case, we see
that weak $C$-\hrcat(\k) implies the $(2C,\k)$-rough 4-point condition.

We prove that (c) implies (a) first for $\k=0$. Suppose $x,y,z$ lie in a
length space $X$ that satisfies the $C'$-rough 4-point condition, and write
$h:=H(x,y,z)$, where $H$ is the standard short function. Suppose also that
$T=T_h(x,y,z)$ is a $h$-short geodesic triangle, $u\in[y,z]_h$, and $v:=x$.

Let $(\by,\bu,\bz,\bx)$ be a $C'$-rough subembedding for
$(x_1,x_2,x_3,x_4)=(y,u,z,x)$. Apply \rf{L:r-uniq-geo}, with $\g$ being the
piecewise linear path from $\by$ to $\g(t):=\bu$ to $\bz$, to get an
associated point $\la(t):=\bu'$ on the line segment $[\by,\bz]$ such that
$|\bu-\bu'|\le \sqrt 3/2$. Thus $d(u,x)\le |\bu'-\bx|+C_1$, where
$C_1:=C'+\sqrt 3/2$.

The Euclidean triangle $T(\bx,\by,\bz)$ satisfies $|\bx-\by|=d(x,y)$ and
$|\bx-\bz|=d(y,z)$, but it is not necessarily a comparison triangle for $T$
because we know only that $|\bz-\by|\ge d(z,y)$. However if we take
$T':=T(x',y',z')$ to be a comparison triangle in $M^2_0$ for $T$, and define
$u'\in[y',z']$ via the equation
\begin{equation}\label{E:u'bu'}
\frac{|u'-y'|}{|z'-y'|} = \frac{|\bu'-\by|}{|\bz-\by|}\,,
\end{equation} then it follows
from \rf{E:plane CAT0} that $|\bu'-\bx|\le |u'-x'|$. Moreover by combining
the subembedding properties, \rf{E:laga1}, \rf{E:laga2}, \rf{E:u'bu'}, and
the fact that $|y'-z'|\le|\by-\bz|$, we see that
\begin{align*}
|u'-y'|&\le |\bu'-\by|\le |\bu-\by|=|u-y|\le t\,, \\
|u'-z'|&\le |\bu'-\bz|\le |\bu-\bz|=|u-z|\le L-t\,,
\end{align*}
and so $u'$ is a comparison point for $u$. Since $d(u,x)\le |\bu'-\bx|+C_1$,
we have shown the $C_1$-rCAT($0$) condition for this choice of data,
comparison triangle, and this particular choice of comparison point $u'$.

A general comparison point $u''$ for $u$ on the side $[y',z']$ of $T'$ must
satisfy $|u''-u'|\le h\le 1$, and so it follows that the $C'$-rough 4-point
condition implies a weak $C$-rCAT($0$) condition for $C=C'+1+\sqrt 3/2$.

It remains to prove that when $-\infty\le\k<0$, (c) implies (a), and both
are equivalent to Gromov hyperbolicity. The $\d$-hyperbolicity condition can
be written in the form
$$ d(x,z)+d(y,w) \le (d(x,y)+d(z,w))\maxo(d(x,w)+d(y,z))+ 2\d\,, $$
and this condition holds (with $\d=\log 3/\sqrt{-\k}$) for all points
$x,y,z,w\in M^2_\k$; see \cite[Theorem~1.5.1]{CDP}. If instead $x,y,z,w$ lie
in a space $X$ that satisfies the $(C',\k)$-rough 4-point condition, then
this condition and the $\d$-hyperbolicity of $M^2_\k$ immediately imply the
$(\d+C')$-hyperbolicity of $X$.

Taking $h\le 1$ in \rf{L:Tripod}, it is readily deduced that every
$\d$-hyperbolic space is $C$-rCAT($-\infty$) for $C=4\d+2$, and so a
fortiori $C$-rCAT($\k$) for every $\k$. We note in particular that the
$(C',\k)$-rough 4-point condition implies the $C$-rCAT($\k$) condition for
$C=4C'+2+4\log 3/\sqrt{-\k}$.
\end{proof}

It follows rather easily from \rf{T:rough 4-pt} that CAT($\k$) spaces are
rCAT($\k$) with roughness constant $C=C(\k)$ when $k<0$; alternatively, this
follows from the well-known {\it geodesic stability} of Gromov hyperbolic
spaces (see for instance \cite[Part III.H]{BH} or \cite{Va}). The fact that
every CAT($\k$) space is an rCAT($\k$) space is also true when $\k=0$: see
\rf{C:CAT0 is rCAT0}.

\begin{rem}
The CAT(0) analogue of \rf{T:rough 4-pt} in \cite[II.3.9]{BH} assumes that
$X$ is a complete space with {\it approximate midpoints}. Such an assumption
readily implies that $X$ is a length space, so we use this latter assumption
in our theorem and in \rf{C:rough 4-pt} below, since we do not wish to
restrict the theory of rCAT(0) spaces to complete spaces.
\end{rem}

It is shown in Bridson and Haefliger \cite[II.3.10]{BH} that CAT(0) is
preserved by various limit operations, including pointed Gromov-Hausdorff
limits and ultralimits; in particular both generalized tangent space and
asymptotic cones of CAT(0) spaces are CAT(0) spaces (see \cite{BH} for the
definition of all of these concepts). The trick is to use the $4$-point
condition and the rather weak limit concept of a $4$-point limit.
Essentially the same arguments, with the $4$-point condition replaced by our
rough $4$-point condition, give us similar results for rCAT(0) spaces which
we now state. We omit the proofs since they are obtained by routine
adjustments to the proofs of II.3.9 and II.3.10 in \cite{BH}. For
completeness, we begin with a definition of $4$-point limits.

\begin{defn}
A metric space $(X,d)$ is a $4$-point limit of a sequence of metric spaces
$(X_n,d_n)$ if for every $x_1,x_2,x_3,x_4\in X$, and $\e>0$, there exist
infinitely many integers $n$ and points $x_{i,n}\in X_n$, $1\le i\le 4$,
such that $|d(x_i,x_j)-d_n(x_{i,n},x_{j,n})|<\e$ for $1\le i,j\le 4$.
\end{defn}

\begin{cor}\label{C:rough 4-pt}
Suppose the length space $(X,d)$ is a $4$-point limit of the weak
$C_n$-\hrcat(0) spaces $(X_n,d_n)$. If $C_n\le C$ for all $n$, then $(X,d)$
is a weak $C'$-rCAT(0) space for some $C'$ dependent only on $C$. If $C_n\to
0$ and $X$ is complete, then $(X,d)$ is a CAT(0) space.
\end{cor}

\begin{cor}\label{C:limit}
Suppose $(X,d)$ is a length space and $(X_n,d_n)$ form a sequence of
$C$-rCAT(0) spaces. Then there exists a constant $C'$ dependent only on $C$
such that:
\begin{enumerate}
\item If $(X,d)$ is a (pointed or unpointed) Gromov-Hausdorff limit of
    $(X_n,d_n)$ then $(X,d)$ is a $C'$-rCAT(0) space.
\item If $(X,d)$ is an ultralimit of $(X_n,d_n)$, then $(X,d)$ is a
    $C'$-rCAT(0) space.
\item If $X$ is rCAT(0), then any asymptotic cone $\text{Cone}_\omega
    X:=\lim_\omega (X,d/n)$ is a CAT(0) space for every non-principal
    ultrafilter $\omega$.
\end{enumerate}
\end{cor}

In each of the cases above, the existence of an approximate midpoint for
arbitrary $x,y\in X$ (meaning a point $m$ such that $d(x,m)\maxo d(y,m)\le
\e+d(x,y)/2$ for fixed but arbitrary $\e>0$) follows easily from the
hypotheses, and so $(X,d)$ is a length space if it is complete.


\section{Rough CAT(0) and roughly unique geodesics} \label{S:rcat0}

In this section, we explore the rough unique geodesic property of (weak)
rough \hcat(0,*) spaces. Recall that CAT(0) spaces are uniquely geodesic.
The rCAT(0) condition for a $h$-short triangle $T(x,y,y)$ readily gives the
following rough version of this.

\begin{obs}Let $x,y$ be a pair of points in a $C$-rCAT($0$) space $(X,d)$,
and let $h:=1/(1\maxo d(x,y))$. Let $\g_i:[0,L_i]\to X$, $i=1,2$ be a pair
of $h$-short segments from $x$ to $y$, parametrized by arclength, with
$L_1\le L_2$. Then $d(\g_1(t),\g_2(t))\le C$, $0\le t\le L_1$.
\end{obs}

The following theorem improves the above observation.

\begin{thm}\label{T:r-uniq-geo} Let $x,y$ be a pair of points in a
weak $C$-\hrcat(0) space $(X,d)$, with $L:=d(x,y)$. For $i=1,2$, let $h_i>0$
and let $\g_i:[0,L+h_i]\to X$ be a $h_i$-short segment from $x$ to $y$,
parametrized by arclength; we assume that $h_1\le h_2$. Then
$$
d(\g_1(t),\g_2(t)) \le 2C + h_2 + \frac{\sqrt{2Lh_1+h_1^2}}{2} +
  \frac{\sqrt{2Lh_2+h_2^2}}{2}\,,\qquad 0\le t\le L+h_1\,,
$$
In particular, if $h_2\le 1/(1\maxo L)$, then
$$ d(\g_1(t),\g_2(t)) \le 2C + 1 + \sqrt{3}\,,\qquad 0\le t\le L+h_1\,. $$
\end{thm}

\begin{proof*}{Proof}
The result follows from the triangle inequality if $t\le h_2$. Let us
therefore assume that $h_2\le L+h_1$ and fix $t\in(0,L+h_1]$. Let
$t':=t\mino L$.

Throughout this proof $i$ can equal either 1 or 2. We write $z_i:=\g_i(t)$
and choose a path $\g_3:[0,L+h]\to X$ from $x$ to $y$, parametrized by
arclength, for some $0<h\le H(x,y,z_1)\mino H(x,y,z_2)$. Let
$T^i=T_h(x,y,z_i)$ be a $h$-short triangle which includes $\g_3$ as a side,
let $\bT^i=T(\bx,\by,\bz_i)$ be corresponding comparison triangles in
$M^2_0$, let $\bu_i$ be a point on $[\bx,\by]$ that is closest to $\bz_i$,
and let $u_i=\g_3(t_i)$, where $t_i:=|\bu_i-\bx|$. Note that $\bu_i$ is a
comparison point for $u_i$.

By basic geometry, we have $t_i\le |\bz_i-\bx|=d(z_i,x)$ and $L-t_i\le
|\bz_i-\by|=d(z_i,y)$, and so $t_i\in[t-h_i,t]$. Thus $|u_1-u_2|\le h_2$.

For $i=1,2$, the concatenation of the two sides of $\bT^i$ other than
$[\bx,\by]$ forms a $h_i$-short path, so by weak $C$-\hrcat(0) and
\rf{L:r-uniq-geo},
\begin{align*}
d(z_1,z_2) &\le d(z_1,u_1)+d(u_1,u_2)+d(u_2,z_2) \\
&\le|\bz_1-\bu_1|+h_2+|\bu_2-\bz_2|+2C \\
&\le 2C+h_2+\frac{\sqrt{2Lh_1+h_1^2}}{2}+\frac{\sqrt{2Lh_2+h_2^2}}{2}\,,
\end{align*}
as required.
\end{proof*}

\begin{rem}\label{R:slight improve}
It is clear from the above proof that the upper bound can be improved if
$h_1\le h$ (where $h$ is as in the proof). In this case, we get
$$
d(\g_1(t),\g_2(t)) \le C + h_2 + \frac{\sqrt{2Lh_2+h_2^2}}{2}\,,\qquad
  0\le t\le L+h_1\,,
$$
and if $h_2\le 1/(1\maxo L)$, then
$$
d(\g_1(t),\g_2(t)) \le C + 1 + \frac{\sqrt{3}}{2}\,,\qquad 0\le t\le L+h_1\,.
$$
\end{rem}

Using the above theorem and remark, we readily get the first statement of
the following corollary. The weak and very weak variants follow by an
examination of the proof of \rf{T:r-uniq-geo}. Alternatively, the weak
rCAT(0) part of this corollary follows from \rf{T:rough 4-pt} (with the same
constant $C''$).

\begin{cor}\label{C:infinitesimal}
A $C$-\hrcat(0) space is $C'$-rCAT(0), for $C'=3C+2+\sqrt{3}$. A weak (or
very weak) $C$-\hrcat(0) space is weak (or very weak) $C''$-rCAT(0), for
$C''=2C+1+\sqrt{3}/2$.
\end{cor}

We now state a variant of \rf{T:r-uniq-geo} for CAT(0) spaces; we omit the
very similar (but less technical) proof.

\begin{thm}\label{T:CAT0-r-uniq-geo} Let $x,y$ be a pair of points in a
CAT($0$) space $(X,d)$, with $L:=d(x,y)$. For $i=1,2$, let $h_i\ge 0$ and
let $\g_i:[0,L+h_i]\to X$ be a $h_i$-short segment from $x$ to $y$,
parametrized by arclength; we assume that $h_1\le h_2$. Then
$$
d(\g_1(t),\g_2(t)) \le h_2 + \frac{\sqrt{2Lh_1+h_1^2}}{2} +
  \frac{\sqrt{2Lh_2+h_2^2}}{2}\,,\qquad 0\le t\le L+h_1\,,
$$
In particular, if $h_1=0$ and $h_2\le 1/(1\maxo L)$, then
$$ d(\g_1(t),\g_2(t)) \le 1 + \frac{\sqrt{3}}{2}\,,\qquad 0\le t\le L\,. $$
\end{thm}

The above theorem has the following easy corollary.

\begin{cor}\label{C:CAT0 is rCAT0}
A CAT(0) space is $C$-rCAT(0) for $C=2+\sqrt{3}$.
\end{cor}

We record here a Rough Convexity lemma for rCAT(0) spaces. This is a rough
analogue of \cite[Proposition II.2.2]{BH}, and can be proved in a similar
way, so we leave its proof as an exercise.

\begin{lem}\label{L:rough cx} Suppose $a_1,a_2,b_1,b_2$ are points in a
$C$-rCAT(0) space. Let $\g_i:[0,1]\to X$ be constant speed $h_i$-short paths
parametrized by arclength from $a_i$ to $b_i$, $i=1,2$, where $h_i=1/(1\maxo
d(a_i,b_i))$. Then
$$ d(\g_1(t),\g_2(t))\le (1-t)d(a_1,a_2)+td(b_1,b_2)+2C\,. $$
If either $a_1=a_2$ or $b_1=b_2$, then we can replace $2C$ by $C$ in the
above estimate.
\end{lem}

\begin{rem}
\label{R:rCAT0-not-inv-qiso} Note that $\R^2$ with the Euclidean metric is
CAT(0), while $\Z^2$ with the $\ell^1$-metric is not even very weak rCAT(0).
Thus rCAT(0) is not invariant under quasi-isometry. By comparison, we note
the well-known facts that Gromov hyperbolicity is invariant under
quasi-isometry in the context of geodesic spaces, while the CAT(0) property
is only invariant under isometry.
\end{rem}


\section{Examples}
\label{S:examples}

We already know that the class of rCAT(0) spaces include both Gromov
hyperbolic (by \rf{T:rough 4-pt} and the fact that rCAT($\k$) implies
rCAT(0) for $\k<0$) and CAT(0) spaces (by \rf{C:CAT0 is rCAT0}). Here we
give two constructions (products and gluing) for getting new rCAT(0) spaces
from old ones, making it easy to construct rCAT(0) spaces that are neither
CAT(0) nor Gromov hyperbolic.

For metric spaces $(X_1,d)$ and $(X_2,d)$, the \emph{$l^2$-product}
$(Z,|\cdot|)$ is given by $X := X_1 \times X_2$ and
$$
d((x_1,x_2),(y_1,y_2)) = \sqrt{(d(x_1,y_1))^2 + (d(x_2,y_2))^2}.
$$
It is well known that $(X,d)$ is a metric space. Note that we are using $d$
to indicate three different metrics: in all cases, the reader should infer
from the context which one is meant. We also use $\len(\g)$ to indicate
length of a path $\g$ in any one of these spaces.

The product of CAT(0) spaces is CAT(0). A proof follows immediately from the
equivalence of CAT(0) with the CN inequality of Bruhat and Tits.  Since it
is not clear if rCAT(0) is equivalent to a rough version of the CN
inequality (that is, bolicity or very weak rCAT(0), as shown in
\rf{P:bolic-vwrCAT(0)}), no such easy proof of the rough analogue of this
result is available. Nevertheless it is true according to the following
theorem. Quantitative dependence of the roughness constant is most neatly
stated using \hrcat(0), but note that this gives quantitative dependence of
the \rcat(0) roughness constant by \rf{C:infinitesimal}.

\begin{thm}
\label{T:product-rCAT0} If $(X_1,d)$ and $(X_2,d)$ are both $C$-\hrcat(0)
spaces, then their $l^2$-product $(X,d)$ is a $(\sqrt{2}\,C)$-\hrcat(0)
space.
\end{thm}

To prove the above theorem, we first need a lemma.

\begin{lem}
\label{L:product-rCAT0} Suppose $(X,d)$ is the $l^2$-product of two length
spaces $(X_1,d)$ and $(X_2,d)$. If $\g=(\g_1,\g_2):[0,T]\to X_1\times X_2$
is a path in $X$, then
\begin{equation} \label{E:l2-path}
\len(\g)\ge\sqrt{(\len(\g_1))^2+(\len(\g_2))^2}\,,
\end{equation}
with equality if $\g_1$ and $\g_2$ are traversed at the same relative rate,
i.e.~if
\begin{equation} \label{E:l2-path-equality}
\len(\g_1)\len\(\g_2|_{[0,t]}\) = \len(\g_2)\len\(\g_1|_{[0,t]}\), \qquad
0<t<T\,.
\end{equation}
\end{lem}

\begin{proof}
The triangle inequality for the Euclidean plane immediately gives the
following inequality for non-negative numbers $a_i$, $b_i$, $1\le i\le n$:
$$
\sum_{i=1}^n \sqrt{a_i^2+b_i^2}\ge
\sqrt{\(\sum_{i=1}^n a_i\)^2+\(\sum_{i=1}^n b_i\)^2}\,.
$$

By taking $a_i:=d(\g_1(t_i),\g_1(t_{i-1}))$ and
$b_i:=d(\g_2(t_i),\g_2(t_{i-1}))$ in the above inequality, where the numbers
$0=t_0\le t_1\le\dots\le t_n=T$ form a partition of $[0,T]$, we deduce
\rf{E:l2-path}.

If $\g_1$ and $\g_2$ are traversed at the same relative rate, then the
vectors $(a_i,b_i)$ defined in the last paragraph are positive scalar
multiples of each other, so we get equality in the planar triangle
inequality, which upon taking a supremum over all such partitions gives
equality in \rf{E:l2-path}.

If the paths are not traversed at the same relative rate then we split $\g$
into two subpaths $\g^i=(\g^i_1,\g^i_2)$, $i=1,2$, where
$\g^1=\g|_{[0,T_1]}$, $\g^2=\g|_{[T_1,T]}$, and $0<T_1<T$ is such that the
equation in \rf{E:l2-path-equality} fails for $t=T_1$. Letting
$a_i=\len(\g^i_1)$ and $b_i=\len(\g^i_2)$, it follows that $(a_1,b_1)$ and
$(a_2,b_2)$ are not scalar multiples of a single vector, and so
\begin{align*}
\len(\g) = \sum_{i=1}^2\len(\g^i) \ge \sum_{i=1}^2\sqrt{a_i^2+b_i^2}
  &>\sqrt{\(\sum_{i=1}^2 a_i\)^2+\(\sum_{i=1}^2 b_i\)^2} \\[.8em]
  &\mathstrut= \sqrt{(\len(\g_1))^2+(\len(\g_2))^2}\,.
\end{align*}
\end{proof}

We are now ready to prove \rf{T:product-rCAT0}. Note that it follows
implicitly from the following proof that the ``if'' clause for equality in
\rf{L:product-rCAT0} is actually an ``if and only if''.

\begin{proof}[Proof of \rf{T:product-rCAT0}]
Suppose $a=(a_1,a_2)$, $b=(b_1,b_2)$ are a pair of points in $X$. Suppose
$\g_i$ is a rectifiable path from $a_i$ to $b_i$ of length $L_i$, $i=1,2$.
By reparametrization if necessary, we assume that $\g_i$ is of constant
speed, and then define $\g=(\g_1,\g_2)$. It follows from
\rf{L:product-rCAT0} that $\len(\g)=\sqrt{L_1^2+L_2^2}$. Since $X_i$ is a
length space, we can choose $\g_i$ so that $L_i$ is arbitrarily close to
$d(a_i,b_i)$, $i=1,2$, and it then follows that $\len(\g)$ is arbitrarily
close to $d(a,b)$. Thus $X$ is a length space.

Letting $a,b\in X$ be as above, it follows from \rf{L:product-rCAT0} that if
$\g=(\g_1,\g_2)$ is a $h$-short path from $a$ to $b$, then $\g_i$ is a
$h'$-short path from $a_i$ to $b_i$, $i=1,2$, where $h'>0$ depends only on
$d(a_1,b_1)$, $d(a_2,b_2)$, and $h$, with $h'\to 0$ as $h\to 0$ (for fixed
$a$, $b$).

Suppose now that we are given points $x,y,z\in X$, a $h$-short triangle
$T:=T_h(x,y,z)$, and points $u,v\in X$ on different sides of $T$. By
projecting this data onto $X_i$, $i=1,2$, it follows that we get an
associated $h'$-short triangle $T_i:=T_h(x_i,y_i,z_i)$, and points
$u_i,v_i\in X_i$ on different sides of $T_i$; here $h'>0$ depends only on
the distances between pairs of vertices of $T_i$, $i=1,2$, and on $h$, with
$h'\to 0$ as $h\to 0$ (for fixed $x,y,z$). We assume that the positive
number $h$, and hence $h'$, is sufficiently small to guarantee that the
$C$-rough CAT(0) condition holds for $T_1$ and $T_2$.

Suppose now that for each side $\g=(\g_1,\g_2)$ of $T$, the projected paths
$\g_1$ and $\g_2$ are traversed at the same relative rate. Let $u$ be on the
side $[x,y]_h$ and let $v$ be on the side $[y,z]_h$. The $C$-rough CAT(0)
condition applied to the projected pairs of points gives
$d(u_i,v_i)\le|\bu_i-\bv_i|+C$, where $\bu_i$, $\bv_i$ are comparison points
for $u_i,v_i$ on the comparison triangle $\BT{i}=T(\bx_i,\by_i,\bz_i)$,
$i=1,2$. It follows readily that if we define $\bT=T(\bx,\by,\bz)$, where
$\bx=(\bx_1,\bx_2)$, etc., if we define $\bu$, $\bv$ analogously, and if we
identify the plane in $\R^4$ containing $\bT$ with $M^2_0$, then $\bT$ is a
comparison triangle for $T$; $\bu$, $\bv$ are comparison points for $u$,
$v$; and the triangle inequality implies that $d(u,v)\le|u-v|+\sqrt{2}\,C$,
as required.

In view of the above, the theorem follows readily once we prove the
following claim: if we fix a pair of points $x=(x_1,x_2),y=(y_1,y_2)\in X$
with $d(x_i,y_i)>0$ for $i=1,2$, and we pick a $h$-short path
$\g=(\g_1,\g_2):[0,1]\to X$ from $x$ to $y$, then $\g_1$ and $\g_2$ are
traversed at \emph{almost the same relative rate}. More precisely, if we
define $L(t;\g_i):=\len(\g_i|_{[0,t]})/\len(\g_i)$, then for all numbers
$0\le t\le 1$ and for our fixed pair of points $x,y$, we claim that there
exists $\e$ dependent only on $h$ such that $|L(t;\g_1)-L(t;\g_2)|<\e$, and
such that $\e\to 0$ as $h\to 0$.

Let $\cF$ be the set of all rectifiable paths from $x$ to $y$, let
$D:=d(x,y)$, let $D_i:=d(x_i,y_i)$, and let $D(t,\g_i)=d(\g_i(t),x_i)$, for
$i=1,2$. Since $\g$ is $h$-short, and so $\g_i$ are $h'$-short, with $h'\to
0$ as $h\to 0$, the claim follows if we prove that $(D(t;\g_1),D(t;\g_2))$
stays uniformly close to the main diagonal of the rectangle
$[0,D_1]\times[0,D_2]$.

Given $\g=(\g_1,\g_2)\in\cF$, we define a path $\la_\g:[0,1]\to[0,1]^2$ by
the equation $\la_\g(t)=(D(t;\g_1),D(t;\g_2))$. Note that $\la_\g$ is a path
from $(0,0)$ to $(D_1,D_2)$ and we need to show that this path remains close
to the diagonal (with a tolerance tending to $0$ as $h\to 0$).

Given $p:=(p_1,p_2)\in[0,1]^2$, let $\cF_p$ be the set of all paths
$\nu\in\cF$ such that $\la_\nu(t)=p$, for some point $t\in[0,1]$. We denote
the associated value of $t$ as $t(p,\nu)$; note that $t(p,\nu)$ may not be
unique, but any non-uniqueness corresponds only to a harmless choice of a
point in a subinterval of $[0,1]$ on which $\nu$ remains stationary so, for
the sake of having a fixed definition, we choose $t(p,\nu)$ to be the
smallest number with the above defining property. Cutting $\nu\in\cF_p$ into
two subpaths $\nu^1,\nu^2$ at the point $t(p,\nu)$, we see that
$$
\len(\nu) = \len(\nu^1)+\len(\nu^2) \ge |(p_1,p_2)|+(D_1-p_1,D_2-p_2)|
  =:f(p)\,.
$$
Note that the function $f:[0,D_1]\times[0,D_2]\to\R$ defined above is
continuous and it takes on its minimum value $|(D_1,D_2)|=D$ only when $p$
lies on the main diagonal of its rectangular domain. By compactness it
readily follows that the minimum value outside any given neighborhood of the
main diagonal is strictly larger than $D$. The claim follows.
\end{proof}

Since the class of rCAT(0) spaces are preserved by taking $l^2$-products, it
is easy to produce an rCAT(0) space that is neither CAT(0) nor Gromov
hyperbolic by taking the $l^2$-product of a Gromov hyperbolic space that is
not CAT(0) and a CAT(0) space that is not Gromov hyperbolic. The simplest
such example is the product of the unit circle and the Euclidean plane.

\bigskip

We now consider spaces obtained by gluing a pair of length spaces
$(X_i,d_i)$, $i=1,2$, along isometric closed subspaces $S_i\subset X_i$,
$i=1,2$ where $f_i:S\to S_i$ are isometries from some fixed metric space
$(S,d_S)$ to $(S_i,d_i|_{S_i})$. This means that we are creating a new space
$X=X_1\sqcup_S X_2$ as the quotient of the disjoint union of $X_1$ and $X_2$
under the identification of $f_1(s)$ with $f_2(s)$ for each $s\in S$. The
{\it glued metric} $d$ on $X$ is defined by the equations $d|_{X_i\times
X_i}=d_i$, $i=1,2$, and
$$
d(x_1,x_2)=\inf_{s\in S} (d_1(x_1,f_1(s))+d_2(f_2(s),x_2))\,, \qquad
  x_1\in X_1,\; x_2\in X_2\,.
$$
Then $d$ is also a length metric \cite[I.5.24]{BH}. For simplicity of
notation, we identify $X_1$, $X_2$, and $S$ with the naturally associated
subspaces of $X$, so that $S=X_1\cap X_2$.

\begin{thm}\label{T:glue1}
If $X=X_1\sqcup_S X_2$ where $(S,d_S)$ is of diameter $D<\infty$ and
$(X_i,d_i)$ is a $C$-rCAT(0) space for $i=1,2$, then $X$ is a $C'$-rCAT(0)
space for some $C'=C'(C,D)$.
\end{thm}

A comparable gluing result for CAT(0) spaces $X_1$ and $X_2$ requires that
$S_i$ be convex in $X_i$ (meaning that it contains all geodesics in $X_i$
between every pair of points in $S_i$) and complete for $i=1,2$, but the
boundedness of $S$ is dropped. The conclusion is then that $X$ is CAT(0);
see \cite[II.11.1]{BH}.

Before proving \rf{T:glue1}, we need some elementary lemmas concerning
planar geometry. The first is a ``small perturbation'' result.

\begin{lem}\label{L:pert}
Suppose $T(x,y,z)$ and $T(x',y',z')$ are triangles in the Euclidean plane,
and that two of $|\;|x-z|-|x'-z'|\;|$, $|\;|y-z|-|y'-z'|\;|$, and
$|\;|x-y|-|x'-y'|\;|$ equals zero, with the third being at most
$h:=1/(1+(|x-y|\maxo|x-z|\maxo|y-z|)^2)$. Suppose also that $u\in[x,z]$,
$u'\in[x',z']$, $v\in[x,y]$, and $v'\in[x',y']$, with $|x-u|=|x'-u'|$ and
$|x-v|=|x'-v'|$. Then $|u'-v|\le |u-v|+2$.
\end{lem}

\begin{proof}
We write $a=|x-z|$, $a'=|x'-z'|$, $b=|y-z|$, $b'=|y'-z'|$, $c=|x-y|$,
$c'=|x'-y'|$, $d=|z-v|$, $d'=|z'-v'|$, $e=|u-v|$, and $e'=|u'-v'|$,
$l=|x-u|$. Since two of the three sidelengths are preserved, we may assume
by symmetry between $y$ and $z$ that $c=c'$. Define the numbers
$s,t,t'\in[0,1]$ by $t=l/a$, $t'=l/a'$, and $s=|x-v|/c$. We assume that
$a\maxo a'\ge 1$, since otherwise the result follows trivially from the
triangle inequality. Thus
$$ |t-t'| \le hl/aa' \le h/(a\maxo a') \le h\,. $$

Using \rf{E:plane CAT0}, we get the following four equations, which we use
implicitly in the rest of the proof:
\begin{align*}
d^2    &= (1-s)a^2 + sb^2 - s(1-s)c^2\,,              \\
(d')^2 &= (1-s)(a')^2 + s(b')^2 - s(1-s)c^2\,,        \\
e^2    &= (1-t)(sc)^2 + td^2 - t(1-t)a^2\,,           \\
(e')^2 &= (1-t')(sc)^2 + t'(d')^2 - t'(1-t')(a')^2\,. \\
\end{align*}
Note that $t(1-t)a^2=l(a-l)$ and similarly $t'(1-t')(a')^2=l(a'-l)$. It is
readily verified that $h\le 1/(1\maxo a'\maxo b'\maxo c')^2$, and trivially
$h\le 1/(1\maxo a\maxo b\maxo c)^2$.

If $a=a'$ and $b'\le b$, then it follows from the above equations that
$d'\le d$ and $e'\le e$, so we are done. If $a=a'$ and $b<b'\le b+h$, we see
that $(d')^2=d^2+2sbh+sh^2\le d^2+3$, and hence that $(e')\le e^2+3$. Thus
$e'\le e+\sqrt 3$ in this case.

Suppose instead that $b=b'$ and $a-h\le a'\le a$. Then $(d')^2\le d^2$, and
so
$$ (e')^2 \le e^2 + (t'-t)d^2 + l(a-a')\le e^2+2\,. $$
In the last inequality, we used the estimate $(t'-t)d^2\le 1$, which in turn
follows from the earlier estimate $|t'-t|\le h$ and the fact that $d\le
a\maxo b$. We deduce that $e'\le e+\sqrt 2$ in this case.

Lastly, suppose that $b=b'$ and $a<a'\le a+h$. Then $(d')^2\le
d^2+(1-s)(ah+h^2)\le d^2+3$, and as in the previous case
$$ (e')^2 \le e^2 + 3t' + (t-t')(sc)^2 \le e^2 + 4\,. $$
Thus $e'\le e+2$ in this case.
\end{proof}

We now state a lemma that we call the {\it Zipper Lemma} because in the
important case $\d_x=\d_y>0$, we get one triangle from another by ``zipping
up'' two sides (shortening them by the same amount).

\begin{lem}\label{L:zipper1}
Suppose $x,y,z,z',u,u'$ are points in the Euclidean plane and write
$\d_x:=|x-z|-|x-z'|$ and $\d_y:=|y-z|-|y-z'|$. Suppose also that
$u\in[x,z]$, $u'\in[x,z']$, and $|x-u|=|x-u'|$. Then
\begin{enumerate}
\item If $v\in[x,y]$ and $|\d_x|\le\d_y$ then $|u'-v|\le|u-v|$.
\item If $v\in[y,z]$ and $v'\in[y,z']$ with $|y-v|=|y-v'|$ and
    $\d_x=\d_y\ge 0$, then $|u'-v'|\le|u-v|$.
\end{enumerate}
\end{lem}

\begin{proof}
By the Cosine Rule applied to the triangles $T(x,u',v)$ and $T(x,z',y)$, it
is clear that the distance from $u'$ to $v\in[x,y]$ decreases as we move
$z'$ directly towards $y$ while keeping $x$, $y$ and $z$ fixed, since both
are associated with the (common) angle at $x$ in both triangles decreasing.
Thus it suffices to prove that the angle at $x$ in the triangle $T(x,z',y)$
is smaller than the angle at $x$ in the triangle $T(x,z,y)$ in the special
cases $\d_y=\d_x>0$ and $\d_y=-\d_x>0$.

We first prove (a) for $\d_x=\d_y>0$. Without loss of generality, we assume
that $x,y$ are given in Cartesian coordinates by $(c,0)$ and $(-c,0)$,
respectively. Let
$$ 2a:= |\;|z-x|-|z-y|\;| = |\;|z'-x|-|z'-y|\;|\,,$$
so that $a\le c$. The lemma is clear if either $a=c$ or $a=0$, so we assume
that $0<a<c$ and write $b=\sqrt{c^2-a^2}$ and $e=c/a$. Thus $z$ and $z'$
both lie on one branch of the hyperbola
$$ \frac{w_1^2}{a^2} - \frac{w_2^2}{b^2} = 1\,,$$
where $(w_1,w_2)$ are the Cartesian coordinates of a point $w$ on this
hyperbola.

We assume for now that $z,z'$ lie on the right branch of this hyperbola,
i.e.~that $|z-x|<|z-y|$. Let $r=|z-x|$ and let $\t$ be the angle at $x$ in
the triangle $T(x,y,z)$. Then $z=(z_1,z_2)$ satisfies the equation
$$ r^2 = (z_1-ae)^2+z_2^2 = (z_1-ae)^2+(e^2-1)(z_1^2-a^2) = (ez_1-a)^2\,, $$
and so $r=ez_1-a$, since we are on the right branch of the hyperbola. Also
$z_1=r\cos(\pi-\t)+ae=ae-r\cos\t$, and so $r=e(ae-r\cos\t)-a$. Rearranging
this equation we get
$$ r = \frac{a(e^2-1)}{1+e\cos\t}\,. $$
It is clear from this equation that the angle $\t$ decreases as $r$
decreases, so we are done.

If instead $|z-x|>|z-y|$, the analysis is similar except that now
$r=a-ez_1$, and so we instead get
$$ r = \frac{a(e^2-1)}{-1+e\cos\t}\,, $$
and again it is clear that $\t$ decreases as $r$ decreases.

We now prove (a) for $\d_y=-\d_x>0$. We could do this in a similar manner to
the proof for $\d_y=\d_x$ above, but using an ellipse rather than a
hyperbola. However we will instead give a slightly shorter calculus proof.
Let $a:=|x-z|$, $b:=|y-z|$, and $c:=|x-y|$, and let $\t$ be the angle at $x$
in the triangle $T(x,y,z)$. The desired conclusion is obvious in the
degenerate cases $b=a+c$ and $c=a+b$, and the degenerate case $a=b+c$ cannot
arise by the triangle inequality since $|x-z'|>a$ and $|y-z'|<b$. We may
therefore assume that we are in the non-degenerate case with $\sin \t>0$.
For the rest of this paragraph prime superscripts indicate derivatives with
respect to a parameter $t$. Specifically, holding $c$ fixed, and considering
$a=a(t)$, $b=b(t)$, and $\t=\t(t)$ to be functions of $t$ with $a'(t)=1$ and
$b'(t)=-1$, it suffices to show that $\t'(t)<0$ for all $0\le t<\d_y$. The
fact that the triangle is non-degenerate at $t=0$ implies that it is
non-degenerate for all $0\le t<\d_y$, and so $\sin \t(t)>0$ on $[0,\d_y)$.
Differentiating the equation $b^2 = a^2+c^2-2ac\cos \t$, we get
$$ -b(t) = a(t) - c(t)\cos \t(t) + a(t)c(t)\sin \t(t) \t'(t)\,, $$
and so
$$ \t'(t) = \frac{-a(t)-b(t)+c(t)\cos \t(t)}{a(t)c(t)\sin \t(t)}\,. $$
The desired inequality $\t'(t)<0$ follows easily for all $0\le t<\d_x$.

Finally we prove (b). Let $a:=|x-z|$, $b:=|y-z|$, and $c:=|x-y|$ as before,
and also let $p:=|z-u|$, $q:=|z-v|$, and $e:=|u-v|$. Without loss of
generality, we assume that $a,b,p,q>0$, $p<a$, and $q<b$. Again we use
calculus and reserve prime superscripts for $t$-derivatives below. Holding
$c$ fixed and taking $a'(t)=b'(t)=p'(t)=q'(t)=-1$, with $\t(t)$ being the
angle at $z$ for $T(x,y,z)$, it suffices to show that $e'(t)<0$.
Differentiating the Cosine Rule for the triangles $T(x,y,z)$ and $T(u,v,z)$
with respect to $t$, we get
\begin{align*}
0         &= (a(t)+b(t))(\cos \t(t)-1) + (a(t)b(t)\sin \t(t)) \t'(t)\,, \\
e(t)e'(t) &= (p(t)+q(t))(\cos \t(t)-1) + (p(t)q(t)\sin \t(t)) \t'(t)\,.
\end{align*}
Combining these equations, we get
$$
\frac{e(t)e'(t)}{p(t)q(t)} =
  \( \frac{1}{a(t)} + \frac{1}{b(t)} - \frac{1}{p(t)} - \frac{1}{q(t)} \)
  (1-\cos \t(t))\,,
$$
and so it is clear that $e'(t)\le 0$, as required.
\end{proof}

We now state a useful perturbation of the previous lemma. The proof is easy:
for (a), first apply \rf{L:pert} to lengthen $|z-x|$ by $h$, and then apply
\rf{L:zipper1}, and for (b), apply \rf{L:pert} twice and then
\rf{L:zipper1}.

\begin{lem}\label{L:zipper2}
Suppose $x,y,z,z',u,u'$ are points in the Euclidean plane and write
$\d_x:=|x-z|-|x-z'|$ and $\d_y:=|y-z|-|y-z'|$. Suppose also that
$u\in[x,z]$, $u'\in[x,z']$, and $|x-u|=|x-u'|$ and we write and
$h:=1/(1+(|x-y|\maxo|x-z|\maxo|y-z|)^2)$. Then
\begin{enumerate}
\item If $v\in[x,y]$ and $|\d_x|\le\d_y+h$ then $|u'-v|\le|u-v|+2$.
\item If $v\in[y,z]$ and $v'\in[y,z']$ with $|y-v|=|y-v'|$ and
    $\d_x+h_1=\d_y+h_2\ge 0$ for some $0\le h_1,h_2\le h$, then
    $|u'-v'|\le|u-v|+4$.
\end{enumerate}
\end{lem}

\begin{proof}[Proof of \rf{T:glue1}]
Let $d$ be the glued metric on $X$. We first claim that any $h$-short path
$\g:[0,L]\to X$ for a pair of points $x,y\in X_1$ lies within a distance
$D/2+2h$ of a $h$-short path for this pair in $X_1$.

Suppose without loss of generality that $\g$ is parametrized by arclength
and not contained in $X_1$. The only parts of $\g$ that do not fully lie in
$X_1$ consist of disjoint subpaths $\g_i$, $i\in I$, where $I$ is a
countable index set and the endpoints of every $\g_i$ lie in $S$.

The distance between these endpoints is the same in either $X_1$ or $X_2$,
and $X_1$ is a length space, so we can replace these subpaths by subpaths in
$X_1$ whose combined length is at most the same as the combined length of
the $\g_i$ subpaths, as long as least one $\g_i$ is non-geodesic, an
assumption that we add for the moment. We therefore get a new $h$-short path
$\g':[0,L']\to X_1$ from $x$ to $y$, parametrized by arclength, with $L'\le
L$. By the triangle inequality, $d(\g(t),\g'(t))\le (D+h)/2+h$ for all $0\le
t\le L'$. This establishes the claim under the assumption that at least one
of the subpaths $\g_i$ is non-geodesic.

The argument when every $\g_i$ is geodesic is similar except that we may not
be able to replace them by geodesic subpaths in $X_1$. As long as
$L<d(x,y)+h$, we can choose the replacement subpaths so short as to
guarantee that the resulting path $\g':[0,L']\to X_1$ from $x$ to $y$ is
$h$-short, is parametrized by arclength, and again satisfies
$d(\g(t),\g'(t))\le (D+h)/2+h$ for all $0\le t\le L$.

The only remaining problem is when $L=d(x,y)+h$. It follows that the parts
of $\g$ other than the $\g_i$ subpaths cannot all be geodesic, so we can
take a non-geodesic subpath of $\g$ that is disjoint from every $\g_i$ and
has length at most $h/2$. We replace this non-geodesic subpath by a shorter
subpath that remains within $X_1$. We now have a path of length less than
$d(x,y)+h$ and we can proceed as in the last paragraph to construct a
$h$-short path $\g':[0,L']\to X_1$ from $x$ to $y$, parametrized by
arclength, satisfying $d(\g(t),\g'(t))\le (D+h)/2+h+h/2$ for all $0\le t\le
L'$. This finishes the proof of the claim.

In view of the above claim and \rf{C:infinitesimal}, it suffices to prove
the rCAT(0) condition for all $h$-short triangles with given vertices
$x,y,z$, where $h\le H$ for some $H=H(x,y,z)>0$, and considering only
$h$-short sides within $X_i$ for any pair of vertices that both lie in
$X_i$, $i=1,2$.

Thus it suffices to prove an rCAT(0) condition for a $h$-short triangle with
vertices $x,y,z$, where $x,y\in X_1$ and $z\in X_2$, and the path in the
triangle from $x$ to $y$ is $\g_{xy}:[0,L_{xy}]\to X_1$, with similar
notation for the other two sides. We assume that $h\le H$, where
$H:=1/3(1+(d(x,y)\maxo d(x,z)\maxo d(y,z))^2)$.

We may further assume that both $\g_{xz}^1:=\g_{xz}|_{[0,M_{xz}]}$ and
$\g_{yz}^1:=\g_{yz}|_{[0,M_{yz}]}$ lie in $X_1$, and both
$\g_{xz}^2:=\g_{xz}|_{[M_{xz},L_{xz}]}$ and
$\g_{yz}^2:=\g_{yz}|_{[M_{yz},L_{yz}]}$ lie in $X_2$, for some choice of
numbers $M_{xz}$ and $M_{yz}$. Of these four subpaths, we call the two with
superscript ``1'' the {\it initial segments} of the associated side of the
triangle, and the other two the {\it final segments} of the associated side.
We write $s_x:=\g_{xz}(M_{xz})$ and $s_y:=\g_{yz}(M_{yz})$.

Symmetry reduces the task of verifying the rCAT(0) condition for points
$u,v$ to the following five cases:
\begin{enumerate}
\item $u$ lies on the initial segment of $\g_{xz}$, and $v$ lies on
    $\g_{xy}$.
\item $u$ lies on the final segment of $\g_{xz}$, and $v$ lies on
    $\g_{xy}$.
\item $u,v$ lie on final segments of $\g_{xz}$ and $\g_{yz}$,
    respectively.
\item $u,v$ lie on initial segments of $\g_{xz}$ and $\g_{yz}$,
    respectively.
\item $u$ lies on the final segment of $\g_{xz}$, and $v$ lies on the
    initial segment of $\g_{yz}$.
\end{enumerate}

In Case (a), we first apply the Zipper Lemma \rf{L:zipper2}(a), with all
data as in that lemma except for the Lemma's $z'$ and $u'$: we take $z'=s_y$
and the $u'$ is taken to be a point on a $h$-short path $\la$ from $x$ to
$s_y$ whose distance to $x$ is $d(x,u)$, if such a point exists (which we
assume for now). Now $d(u,v)\le d(u,u')+d(u',v)$, and by the rCAT(0)
condition for the triangle with vertices $x,s_x,s_y$, we see that
$d(u,u')\le D+C$. Combining the rCAT(0) conditions for the triangles with
vertices $x,y,s_y$ with the Zipper Lemma and this estimate for $d(u,u')$, we
deduce the desired rCAT(0) inequality for the pair $u,v$ in the triangle
with vertices $x,y,z$.

If there is no point $u'$ on $\la$ with $d(u',x)=d(u,x)$, then take
$u'=s_y$, and so $d(u',x)<d(u,x)$. As in the last paragraph, we get an
rCAT(0) inequality for the pair $u'',v$, where $u''$ is a point on $\g_{xy}$
such that $d(u'',x)=d(u',x)$. But since
$$ d(u',x) \ge d(s_x,x) - D \ge d(u,x) - D - h\,, $$
and so $d(u',u)\le D+2h$. Since this quantity is bounded, the rCAT(0)
condition for $u'',v$ implies an rCAT(0) condition for $u,v$ (with a
parameter $C$ that is larger by $2D+4h$).

We next consider Case (b). First construct a ``comparison quadrilateral''
$\bQ$ with vertices $\bx,\by,\bs_y,\bs_x$ for the quadrilateral $Q$ with
vertices $x,y,s_y,s_x$. \rf{T:rough 4-pt} ensures that we can do this in a
certain sense, but we need less than guaranteed by that: in fact we need
only that distances between each of the four pairs of adjacent pairs of
adjacent vertices is preserved (such a ``comparison quadrilateral'' exists
for any quadrilateral in any metric space). We form a new metric space space
$(G,d_G)$ by gluing a filled Euclidean triangle with sides of length
$|\bs_x-\bs_y|=d(s_x,s_y)$, $d(y,z)-d(y,s_y)$, and $d(x,z)-d(x,s_x)$, to the
Euclidean plane along the line segment from $\bs_x$ to $\bs_y$. If we can
prove a variant rCAT(0) condition for the triangle with vertices $x,y,z$,
and $u,v$ as in Case (b) where we have all the usual inequalities and
equations of \rf{D:comp-tri}, but with the (geodesic) comparison triangle
$\bT$ in $G$ rather than the Euclidean plane, then the usual CAT(0)
condition follows by combining this variant rCAT(0) condition with the usual
CAT(0) condition for the comparison triangle in $G$; the fact that $G$ is
CAT(0) follows from the CAT(0) gluing theorem referred to after the
statement of \rf{T:glue1}. Since $u$ is on the final segment of $\g{xz}$ and
$v\in X_1$, we see that $d(u,v)=d(u,s)+d(s,v)$ for some $s\in S$, and so
$d(u,v)$ is within a distance $2D$ of $d(u,s_x)+d(s_x,v)$. Similarly if
$\bu,\bv$ are the comparison points for $u,v$ in $\bT$, then $d(\bu,\bv)$ is
within a distance $2D$ of $d(\bu,\bs_x)+d(bs_x,\bv)$. Since $d(u,s_x)$ and
$d(\bu,\bs_x)$ differ by at most $h$, it follows that the desired variant
rCAT(0) condition for $u,v$ follows from the usual rCAT(0) condition for the
pair of points $s_x,v$, as proven in Case (a) (once we increase the
parameter $C$ by $8D+2h$).

Case (c) follows easily from the fact that $u,v$ lie on a $h$-short triangle
in $X_2$ with vertices $z$, $u_0$, and $v_0$, with $d(u_0,v_0)\le D$; we
leave the details to the reader.

We next handle Case (d). Suppose first that we can find a point $w\in X_1$
such that $d(w,s_x)\le D+h$, $d(w,s_y)\le D+h$, and
$$ d(z,x)-d(w,x)+h_1 = d(z,y)-d(w,y)+h_2 \ge 0\,, $$
for some $0\le h_1,h_2\le 3h$. We pick $h$-short paths $\g_{xw}$ and
$\g_{yw}$ from $x$ to $w$, and from $y$ to $w$, respectively, and associated
points $u'$ on $\g_{xw}$ and $v'$ on $\g_{yw}$ such that $d(u',x)=d(u,x)$
and $d(v',y)=d(v,y)$ (as for $u''$ in Case (a), we let $u'$ and/or $v'$
equal $w$ if one or other of these last equations cannot be satisfied).
Applying \rf{L:zipper2}(b) with $3h$ playing the role of $h$ in the lemma,
it is clear that the the distance apart of the comparison points for $u,v$
in the comparison triangle $T_1=T(\bx,\by,\bz)$ for the triangle with
vertices $x,y,z$ is either larger, or smaller by at most $4$, than the
distance apart of the comparison points for $u,v$ in the comparison triangle
$T_2=T(\bx,\by,\bw)$ for the triangle with vertices $x,y,w$, assuming that
we choose the comparison points so that $d(u,x)=|\bu-\bx|$, and similarly
preserve $d(u',x)$, $d(v,y)$, and $d(v',y)$. Putting the rCAT(0) condition
for the pair $u,u'$ together with the estimates $d(u,u')\le D+h+C$ and
$d(v,v')\le D+h+C$, we get an rCAT(0) condition for $u,v$, as required.

It remains to find a point $w$ with the desired properties. Let
$\la:[0,L]\to X_1$ be a $h$-short path from $s_x$ to $s_y$, parametrized by
arclength. Let $w=\la(t)$ for some $t\in[0,L]$. Then $d(w,s_x)\le D+h$ and
$d(w,s_y)\le D+h$ whenever $w=\la$. Also let $\d_x:=d(x,z)-d(x,w)$ and
$\d_y:=d(y,z)-d(y,w)$.

When $t=0$, the $h$-shortness of $\g_{xz}$ implies that
$$ L_{xz}-M_{xz}+h \le \d_x+2h \le L_{xz}-M_{xz}+2h\,, $$
and the $h$-shortness of $\g_{xz}$ implies that
$$
-D\le -d(s_y,w) \le h+d(z,y)-d(z,s_y)-d(s_y,w) \le \d_y+h \le
  L_{xz}-M_{xz}+h\,.
$$
In particular, $\d_x+2h\ge\d_y$. Similarly when $t=L$ we get that
$\d_y+2h\ge\d_x$. It follows that for some $t\in[0,L]$ we have
$\d_x+h_1=\d_y+h_2$, for some non-negative numbers $h_1,h_2$ with
$h_1+h_2=3h$. (In fact we get $\d_x+h_1=\d_y+h_2$ for some non-negative
numbers $h_1,h_2$ satisfying $h_1+h_2\le 2h$ and $h_1h_2=0$, but it suits us
to increase both numbers so that $h_1+h_2=3h$.)

Note that
$$
L\le d(s_x,s_y)+h\le d(s_x,z)+d(z,s_y)+h\le L_{xz}-M_{xz}+L_{yz}-M_{yz}+h\,,
$$
and so
$$
d(w,x))+d(w,y) \le (M_{xz}+t)+(L-t+M_{yz}) \le L_{xz}+L_{yz}+h\le
  d(z,x)+d(z,y)+3h\,.
$$
It follows that $\d_x+\d_y+3h\ge 0$, and so $\d_x+h_1=\d_y+h_2\ge 0$, as
required.

Case (e) follows from Case (d) in the same way as Case (b) follows from Case
(a).
\end{proof}

It is often useful to glue an infinite number of spaces together, sometimes
along a single point or set, or sometimes at different places along some
base space. The following general gluing theorem says that for either of
these types of gluing of $C$-rCAT(0) spaces along uniformly bounded gluing
sets, we get another rCAT(0) space.

\begin{thm}\label{T:glue2}
Suppose we have a collection of $C$-rCAT(0) spaces $X_i$, $i\in I$, where
$I$ is some index set containing $0$ as an element. We write
$I^*=I\setminus\{0\}$. Suppose further that in each $X_i$, $i\in I^*$, we
have a closed subspace $S_i$ that is glued isometrically to a closed
subspace $T_i$ of $X_0$. Suppose further that $S_i$ (and $T_i$) is of
diameter at most $D<\infty$, $i\in I^*$. Then the resulting space $X$ is a
$C'$-rCAT(0) space for some $C'=C'(C,D)$.
\end{thm}

\begin{proof}[Sketch of proof]
Using a similar argument to the proof of the claim at the beginning of the
proof of \rf{T:glue1}, we see that for sufficiently small $h$, a $h$-short
path between $x\in X_i$ and $j\in X_j$, $i,j\in I$, is within a bounded
Hausdorff distance of a $h$-short path path that only passes through $X_i$,
$X_j$, and $X_0$. Thus we may restrict ourselves to examining $h$-short
triangles whose sides are of this type, and an rCAT(0) condition for any
pair of points on such a triangle with vertices in $X_i$, $X_j$, and $X_k$
follows from at most three appeals to \rf{T:glue1} (to glue $X_i$, $X_j$,
and $X_k$ to $X_0$).
\end{proof}

As mentioned earlier, if we glue a pair of CAT(0) spaces along a pair of
isometric convex subspaces, we get a CAT(0) space. It is tempting therefore
to suspect that if we glue a pair of rCAT(0) spaces along a pair of
isometric convex subspaces (or even isometric ``roughly convex'' subspaces,
whatever this should mean), we get an rCAT(0) space. However the following
example shows that this is false.

\begin{exa}\label{X:glue}
First let $X_1,X_2$ be two disjoint isometric copies of the closed Euclidean
upper half-plane $X:=\{(x,y): y\ge 0\}$. We write $I_i:X\to X_i$, $i=1,2$,
for the natural (isometric) identification maps. We will isometrically glue
the $y=0$ edges of $X_1$ and $X_2$ to the edges given by opposite sides of a
``warped ladder'' $Y$.

To construct $Y$, we begin with its two ``sides'' consisting of two disjoint
copies of $\R$: for $i=1,2$, let $\nu_i:\R\to Y$ be isometric maps to these
sides of $Y$. Next we define the {\it rung} $R_n$, $n\in\Z$, of the ladder
$Y$ to be a line segment of length $\exp(-|n|)$, with one endpoint glued to
$\nu_1(n)$ and the other to $\nu_2(n)$. These line segments $R_n$ are
pairwise disjoint, and disjoint from the lines $\nu_i(\R)$ except where
glued at their endpoints. We give $Y$ the glued metric $d_Y$. Thus adjacent
rungs are always a distance $1$ apart but the two sides are warped in the
sense that the distance from $\nu_1(n)$ to $\nu_2(n)$ decays exponentially
in $|n|$.

We now isometrically glue together $X_1$, $Y$, and $X_2$ along their edges
to get the glued space $Z=(X_1\sqcup_\R Y)\sqcup_\R X_2$. More precisely,
for each $x\in\R$ and $i=1,2$, we identify $I_i(x,o)$ with $\nu_i(x)$. We
denote the glued metric by $d$. As usual, it is convenient to consider
$X_1,X_2,Y$ to be subsets of $Z$.

Rather trivially $X_1,X_2$ are CAT(0), and so rCAT(0). Also $Y$ is rCAT(0):
probably the easiest way to see this is to note that $Y$ is roughly
isometric to $\R$ and so Gromov hyperbolic. Thus $Z$ is obtained by
isometrically gluing three rCAT(0) spaces along convex subsets and, if the
isometric gluing of two rCAT(0) spaces along a closed convex set were always
rCAT(0), then $Z$ would be rCAT(0) (just apply such a result twice).

However we claim that $(Z,d)$ is not rCAT(0). To see this note first that
$$ d_n(y):=d(I_i((0,y)),I_i((n,0))) = \sqrt{y^2+n^2} $$
is independent of $i$ and is an even function of $n\in\Z$, and the unique
geodesic of length $d_n$ in $Z$ is a line segment in $X_i$ between these
points.

Clearly $Z$ is proper, and $d_n(\cdot)$ is increasing and unbounded as a
function of $|n|$, so we readily deduce that there exists a geodesic segment
from $I_1((0,y))$ to $I_2((0,y))$ for all $y>0$. Moreover
$$ d_n(y)-d_0(y) = \sqrt{y^2+n^2}-y\to 0 \qquad (y\to\infty)\,. $$
Fixing $n\in\N$ and choosing $y$ so large that $d_n(y)-d_0(y) <
e^{-n+1}-e^{-n}$, we ensure that one geodesic segment between $I_1((0,y))$
and $I_2((0,y))$ in $Z$ must cross $Y$ along a rung $R_N$ for some $N\ge n$.
By symmetry, another geodesic segment between these two points goes via
$R_{-N}$. Letting $n\to\infty$, we therefore have a pair of geodesic
segments with the same endpoints such that the distance between their
midpoints is greater than $2n$, and so can be arbitrarily large. Such a
configuration is incompatible with the rCAT(0) condition.
\end{exa}

\bigskip

Finally, we show that there are no interesting examples among the class of
normed real vector spaces. As is well known, such spaces are CAT(0) if and
only if they are inner product spaces \cite[II.1.14]{BH}. It is
straightforward to use the dilation structure of such spaces to show that
they must be CAT(0) if they are rCAT(0); we give the details for
completeness.

\begin{prop}\label{P:VS}
Suppose $(V,\|\cdot\|)$ is a normed real vector space with distance
$d(x,y)=\|x-y\|$. Then $V$ is rCAT(0) if and only if it is CAT(0).
\end{prop}

\begin{proof}
Suppose $(V,d)$ is $C$-rCAT(0). Being a normed vector space, $V$ is
certainly a geodesic space. We wish to prove the CAT(0) condition for a
fixed but arbitrary geodesic triangle $T$ with vertices $x,y,z\in V$. The
translation invariance of $d$ allows us to assume without loss of generality
that $x=0$. Let $\bT$ be a comparison triangle in $M^2_0$ with vertex at $0$
corresponding to $x=0$, let $u,v$ be points on different sides of $T$ and
let $\bu,\bv$ be the respective comparison points on $\bT$.

We now exploit the dilation invariance of $V$. Given a geodesic $\g:[0,L]\to
V$ from $a\in V$ to $b\in V$, we get a dilated geodesic $R\g:[0,L]\to V$
from $Ra$ to $Rb$ for any given $R>0$ by defining $(R\g)(t)=R\g(t)$. If we
dilate our geodesic triangle $T$ in this manner, we get a geodesic triangle
which we call $RT$, and it is clear that the similarly dilated Euclidean
triangle $R\bT$ is a comparison triangle for $RT$, and that $R\bu,R\bv$ are
respective comparison points for $Ru,Rv\in RT$. Furthermore if
$d(u,v)=|\bu-\bv|+\e$ for some $\e>0$, then $d(Ru,Rv)=|R\bu-R\bv|+R\e$, so
by taking $R>C/\e$ we contradict the rCAT(0) inequality. Thus the rCAT(0)
condition can only hold if the CAT(0) condition holds.
\end{proof}

\begin{rem}
\label{R:lp-prod} It follows from the above theorem that we cannot change
the $l^2$-product in \rf{T:product-rCAT0} to an $l^p$-product for any $p\ne
2$, since certainly the $l^p$-product of two Euclidean lines is rCAT(0) only
when $p=2$.
\end{rem}


\end{document}